\documentclass[reqno]{amsart}
\usepackage{hyperref}
\usepackage[usenames,dvipsnames]{xcolor}
\hypersetup{%
	colorlinks=true, linkcolor=blue,
	citecolor=ForestGreen
}
\usepackage[paper=letterpaper,margin=.9in]{geometry} 
\usepackage{amsmath,amssymb,mathrsfs,amsthm,amsfonts,mathtools}
\usepackage[inline]{enumitem} 
\usepackage[utf8]{inputenc}
\usepackage{lmodern}
\usepackage{array}
\usepackage{adjustbox}
\usepackage{graphicx}
\usepackage[english]{babel}
\usepackage{float} 
\usepackage{comment}
\usepackage{tikz}
\usepackage{bbm}
\usetikzlibrary{decorations.markings}
\usetikzlibrary{calc,angles,positioning}
\usetikzlibrary{decorations.pathreplacing}
\usepackage{chngcntr}
\usepackage{apptools}
\theoremstyle{plain}
\newtheorem{theorem}{Theorem}[section]

\newtheorem{defin}[theorem]{Definition}
\newtheorem{prop}[theorem]{Proposition}

\newtheorem{lemma}[theorem]{Lemma}

\newtheorem{remark}[theorem]{Remark}

\newcommand{\dd}{\delta_p}
\newcommand\numberthis{\addtocounter{equation}{1}\tag{\theequation}}

\newcommand{\phinsig}{\phi_{e^{-t\sigma}, t}^{(n)}}
\newcommand{\Rp}{\mathbb{R}_{>0}}

\newcommand{\Z}{\mathbb{Z}}
\newcommand{\Zn}{\Z_{\geq 0}}
\newcommand{\Zp}{\Z_{\geq 1}}
\newcommand{\phf}{p^{\text{hf}}}
\newcommand{\II}{\mathcal{I}}
\newcommand{\IIt}{\widetilde{\mathcal{I}}}
\newcommand{\KK}{\mathcal{K}}
\newcommand{\JJ}{\mathcal{J}}
\newcommand{\R}{\mathbb{R}}
\newcommand{\RL}{\R^L}
\newcommand{\row}{\mathsf{r}}
\newcommand{\Be}{\mathsf{B}}

\newcommand{\M}{\mathfrak{M}}

\newcommand{\lt}{\exp\big(-s \ZZhf(2t, 0) e^\frac{t}{12}\big)}
\renewcommand{\tt}{t}
\newcommand{\Ai}{\text{Ai}}
\renewcommand{\AA}{\mathcal{A}_p}
\newcommand{\MMone}{\mathsf{M}_1}
\newcommand{\MMtwo}{\mathsf{M}_2}
\newcommand{\AAA}{\mathsf{A}} 
\newcommand{\AAp}{\mathcal{A}'_p}
\newcommand{\AApp}{\mathcal{A}''_p}
\newcommand{\AAone}{\AAA_1}
\newcommand{\AAtwo}{\AAA_2}

\newcommand{\D}{\mathsf{D}} 
\newcommand{\f}{F}
\newcommand{\pfll}{\text{Pf}\big[K(x_i, x_j)\big]_{i, j = 1}^L}
\newcommand{\ft}{\f_{\frac{1}{3}, 2}}
\newcommand{\ftt}{\f_{\frac{2}{3}, 2}}
\newcommand{\BB}{\mathcal{B}}
\newcommand{\BBpL}{\BB_{p, L}}
\newcommand{\As}{\mathsf{E}_{1}}
\newcommand{\Ass}{\mathsf{E}_{2}}
\newcommand{\qq}{\mathsf{q}}

\newcommand{\C}{C}
\newcommand{\EEp}{\mathsf{E}'}
\newcommand{\UU}{U}
\newcommand{\VV}{V}

\newcommand{\GOE}{F_{\text{GOE}}}
\newcommand{\HHhf}{\mathcal{H}^{\text{hf}}}

\newcommand{\ZZhf}{\mathcal{Z}^{\text{hf}}}

\begin{document}
	\numberwithin{equation}{section}
	\title{Lyapunov exponents of the half-line SHE}
	\author{Yier Lin}
	\address{Y.\ Lin,
		Department of Mathematics, Columbia University,
		2990 Broadway, New York, NY 10027}
	\email{yl3609@columbia.edu}
	\begin{abstract}
		We consider the half-line stochastic heat equation (SHE) with Robin boundary parameter $A = -\frac{1}{2}$.
		Under narrow wedge initial condition, we compute every positive (including non-integer) Lyapunov exponents of the half-line SHE. As a consequence, we prove a large deviation principle for the upper tail of the half-line KPZ equation under Neumann boundary parameter $A = -\frac{1}{2}$ with rate function $\Phi_+^{\text{hf}} (s) = \frac{2}{3} s^{\frac{3}{2}}$. This confirms the prediction of  \cite{KLD18, MV18} for the upper tail exponent of the half-line KPZ equation. 
	\end{abstract}
	\maketitle
	\section{Introduction}
 In this paper, we study the half-line KPZ equation, namely the KPZ equation on $\mathbb{R}_{\geq 0}$, with Neumann boundary parameter $A$. Introduced in \cite{CS18}, the equation is formally written as  
\begin{equation}\label{eq:hfkpz}
\begin{cases}
&\partial_t \HHhf (t, x) = \frac{1}{2} \partial_{xx} \HHhf (t, x) + \frac{1}{2} (\partial_{x} \HHhf (t, x))^2 + \xi(t, x),\\
&\partial_x \HHhf (t, x)\Big|_{x = 0} = A,
\end{cases}
\end{equation}
where $\xi(t, x)$ is the Gaussian space time white noise. The solution theory of \eqref{eq:hfkpz} is ill-posed due to the non-linearity and the space-time white noise. One way to properly define the solution is to consider the \emph{Hopf-Cole solution} $\HHhf(t, x) := \log\ZZhf (t, x)$ where $\ZZhf$ solves the half-line stochastic heat equation (SHE) with Robin boundary parameter $A$, i.e.
\begin{equation*}
\begin{cases}
&\partial_t \ZZhf (t, x) = \frac{1}{2} \partial_{xx} \ZZhf (t, x) + \frac{1}{2} \ZZhf (t, x) \xi(t, x),\\
&\partial_x \ZZhf (t, x)\Big|_{x = 0} = A \ZZhf (t, 0).
\end{cases}
\end{equation*}
We say $\ZZhf$ is a solution to the half-line SHE if for every $t> 0$, $\ZZhf(t, \cdot)$ is adapted to the sigma algebra generated by $\ZZhf(0, \cdot)$ and the space-time white noise up to time $t$, and satisfies the mild formulation 
\begin{equation*}
\ZZhf(t, x)  = \int_{\mathbb{R}_{\geq 0}} p^{\text{hf}}_t(x, y) \ZZhf(0, y) dy + \int_0^t \int_{\mathbb{R}_{\geq 0}} p_{t-s}^{\text{hf}}(x, y) \mathcal{Z}^{\text{hf}}(s, y) dyds,
\end{equation*} 
For fixed $x \in \mathbb{R}_{\geq 0}$, $\phf_t (x, y)$ satisfies the half-line heat equation $\partial_t \phf_t (x, y) = \frac{1}{2} \partial_{yy} \phf_t (x, y)$ for all $y > 0$ with boundary condition $\phf_0 (x, y) = \delta_{x} (y)$ and $\partial_x \phf(t, 0) = -A \phf(t, 0)$. \cite{CS18} proves the existence, uniqueness and positivity of $\ZZhf$ for non-negative boundary parameter $A$ and later \cite{Par19} extends these results to the scope of all $A \in \mathbb{R}$. As a consequence,  the Hopf-Cole solution $\HHhf(t, x) = \log \ZZhf(t, x)$ is well-defined. Note that the solution to \eqref{eq:hfkpz} can also be formulated in other different but equivalent ways, see \cite{GPS17, GH19}.
\bigskip
\\
The half-line KPZ equation plays an important role characterizing how the surface grows subject to a boundary. In addition, it is the (weak) scaling limit of various half space models lying in the half-space KPZ universality \cite{Wu18,CS18,Par19a}. Interestingly, such half-space random growth models usually exhibit a phase transition depending on the strength of repulsion/attraction at the boundary, which is characterized by the boundary parameter. Such phase transition is related to wetting/depinning transition which goes back to \cite{Kar85, Kar87} and was proved for various discrete half-space models \cite{BR01,IS04,BBCS18}.  
\bigskip
\\
For the half-line KPZ equation, we restrict ourselves in a particular initial condition called \emph{narrow wedge initial condition}, which corresponds to setting $\ZZhf(0, x)$ to be a Dirac-delta function at zero. It is widely believed that the fluctuation of $\HHhf(2t, 0)$ at late time exhibits a phase transition at $A = -\frac{1}{2}$ \cite[Conjecture 1.2]{Par19}. More precisely, the fluctuation of $\HHhf(2t, 0)$ will be Gaussian/Tracy-Widom GOE/GSE \cite{TW96} depending on whether the boundary parameter $A$ is smaller than, equal to or larger than $-\frac{1}{2}$. 
\bigskip
\\
\cite[Theorem 1.1]{Par19} (also see \cite[Remark 1.1]{BBCW18}) shows that when $A = -\frac{1}{2}$, for all $s \in \mathbb{R}$,
\begin{equation}\label{eq:clt}
\lim_{t \to \infty} \mathbb{P}\Big(\frac{\HHhf(2t, 0) + \frac{t}{12}}{t^{\frac{1}{3}}} \leq s\Big) = \GOE(s),
\end{equation} 
where $\GOE(s)$ is distribution function of the Tracy-Widom GOE distribution \cite{TW96}. The key ingredient to arriving at \eqref{eq:clt} is the exact formula for the Laplace transform of $\ZZhf(2t, 0) + \frac{t}{12}$ developed in \cite{BBCW18, Par19}, see Theorem \ref{thm:exactformula}.  The conjectured Gaussian/GSE fluctuation for $A < -\frac{1}{2}$ and $A > -\frac{1}{2}$ are supported in a non-rigorous way by the works \cite{GLD12,BBC16,DJKPT19, KLD20}. In the GSE region $A > -\frac{1}{2}$, an exact formula of the Laplace transform of $\ZZhf(2t, 0)$ is also conjectured in the aforementioned works. 
\bigskip
\\
In this paper, we focus on the critical regime $A = -\frac{1}{2}$.
\bigskip
\\
Having considered the limit theorem \eqref{eq:clt}, it is natural to think about the large deviation principle (LDP), i.e. the probability that $\HHhf(2t, 0) + \frac{t}{12}$ deviates from zero in a size of $t$, as $t \to \infty$. It is expected that for $s > 0$,
\begin{align}\tag{lower tail}
&-\lim_{t \to \infty} \frac{1}{t^2} \log \mathbb{P}\Big(\HHhf(2t, 0) + \frac{t}{12} < -st\Big) = \Phi_{-}^{\text{hf}} (s)\\
\tag{upper tail}
&-\lim_{t \to \infty} \frac{1}{t} \log \mathbb{P}\Big(\HHhf(2t, 0) + \frac{t}{12}  > st\Big) = \Phi^{\text{hf}}_+ (s).
\end{align} 
Note that the upper and lower tail LDP have different speeds ($t$ vs $t^2$). One way to explain such phenomenon is to view $\HHhf(2t, 0)$ as the free energy of a half-space continuum directed random polymer with a wall at $x = 0$. For various discrete/continuum polymers, the $t$ vs $t^2$ phenomenon is observed  and explained in \cite{LDMS16, BGS17, DT19}. Here, let we provide a different explanation. If we replace $s$ with $t^{\frac{2}{3}} s$ in \eqref{eq:clt}, the right hand side of \eqref{eq:clt} becomes $\GOE(t^{\frac{2}{3}} s)$. Since Tracy-Widom GOE distribution has left and right tail: as $s \to \infty$
$\GOE (-s) \sim \exp(-\frac{s^3}{24}), 1-\GOE(s) \sim \exp(-\frac{2}{3} s^{\frac{3}{2}})$, see \cite{TW09}. Hence,
we recover the $t^2$ and $t$ speed of LDP for the lower and upper tail. \cite[Corollary 1.3]{Tsai18} proves the LDP for the lower tail 
and identifies the rate function $\Phi_-^{\text{hf}}(s)$.
\bigskip
\\
In this paper, we prove that the upper tail LDP holds with $\Phi_+^{\text{hf}} (s) = \frac{2}{3} s^{\frac{3}{2}}$. Note that this is the first rigorous result concerning the upper tail LDP of the \emph{half-space models} in the KPZ universality class. The $\frac{3}{2}$-exponent for the upper tail also arises in the work of \cite{KLD18, MV18}, where the LDP for half-line KPZ equation at short time was studied.
\bigskip
\\
The upper tail LDP of the half-line KPZ equation is closely related to the \emph{Lyapunov exponent} of the half-line SHE.  More precisely, for $p \in \mathbb{R}_{> 0}$ we call the $p$-th Lyapunov exponent of the SHE to be limit of $t^{-1} \log\mathbb{E}\big[\ZZhf(2t, 0)^p\big]$ as $t \to \infty$. If such limit exists for every $p$, in the spirit of G\"{a}rtner-Ellis theorem, $\HHhf(2t, 0) = \log\ZZhf(2t, 0)$ satisfies a LDP with rate function to be the Legendre-Fenchel transform of the Lyapunov exponents (as a function of  $p$). We remark that the Lyapunov exponents also capture the nature of intermittency, which is a universal property for the random fields with multiplicative noise and has been studied extensive in the literature \cite{GM90, CM94, GKM07, FK09, CJK13, CJKS13, CD15, Chen15, BC16, KKX17}.
\subsection{Main result and proof idea}
\label{sec:mainresult}
From now on, we use $\ZZhf (2t, 0)$ to denote the solution to the half-line SHE with Robin boundary parameter $A = -\frac{1}{2}$ and Dirac-delta initial data $\ZZhf (0, x) = \delta_{x = 0}$. 
Our main contribution is rigorously computing the Lyapunov exponents of the half-line SHE.
\begin{theorem}[Lyapunov exponents and upper tail LDP]\label{thm:main}
	We have
	\begin{enumerate}[leftmargin = 20pt, label = (\roman*)]
		\item \label{item:thm1} For every $p \in \mathbb{R}_{>0}$, one has
		$
		\lim_{t \to \infty} \frac{1}{t} \log \mathbb{E}\big[\ZZhf(2t, 0)^p \exp({\frac{pt}{12}})\big] = \frac{p^3}{3}.
		$
		\item \label{item:thm2} For every $s \in \mathbb{R}_{>0}$, one has the upper tail LDP:
		$
		-\lim_{t \to \infty} \frac{1}{t}\log \mathbb{P}\big(\HHhf(2t, 0) + \frac{t}{12} > st\big) = \Phi_+^{\text{hf}} (s) = \frac{2}{3} s^{\frac{3}{2}}.
		$
	\end{enumerate}
\end{theorem}
\begin{remark}
The above upper tail LDP rate function matches with the right tail of the GOE, which is the limiting distribution in \eqref{eq:clt}. Such matching between the upper tail LDP rate function of KPZ equation  and the right tail of the limiting Tracy GUE/GOE/Baik-Rains distribution has been predicted in \cite{LDMS16, LDMRS16, MS17} and has been confirmed in various situations \cite{DT19, GL20}.
\end{remark}
 Let us briefly explain the idea for the proof of Theorem \ref{thm:main}. A more detailed discussion will be given in Section \ref{sec:overview}. First of all, it is not hard to get Theorem \ref{thm:main} \ref{item:thm2} once we obtain  \ref{item:thm1}, see Proposition 1.12 of \cite{GL20}. Hence, we focus on the proof of \ref{item:thm1}.   One crucial input is the following exact formula obtained in \cite[Theorem 1.3]{Par19} (also see \cite[Theorem 7.6]{BBCW18}), which is stated as follows.
\begin{theorem}[Theorem 1.3 of \cite{Par19}, Theorem 7.6 of  \cite{BBCW18}] \label{thm:exactformula}
When $A = -\frac{1}{2}$, we have for all $s \geq 0$,
\begin{equation}\label{eq:exactformula}
\mathbb{E}\Big[\lt\Big] = \mathbb{E}\Big[\prod_{i=1}^{\infty} \frac{1}{\sqrt{1 + 4s \exp(t^{\frac{1}{3}} \mathbf{a_k})}}\Big],
\end{equation}
where $\mathbf{a}_1 > \mathbf{a}_2 > \cdots $ is the GOE-Airy point process defined in Definition \ref{def:GOE}. 
\end{theorem}
Our argument for proving Theorem \ref{thm:main} \ref{item:thm1} follows \cite{DT19} at the beginning. Firstly, we write the $p$-th moment of $\ZZhf(2t, 0)e^{\frac{t}{12}}$ in terms of the Laplace transform (Lemma \ref{lem:momentlaplace})
\begin{equation*}
\mathbb{E}\Big[\big(\ZZhf(2t, 0) e^{\frac{t}{12}}\big)^p \Big] = \frac{(-1)^n}{\Gamma(1-\alpha)}\int_0^{\infty} s^{-\alpha} \partial^n_s \Big(\mathbb{E}\big[\lt\big]\Big) ds, \qquad  n = \lfloor p \rfloor + 1, \alpha = p + 1 - n.  
\end{equation*}
Decompose the integral region into  $(0, 1]$ and $(1, \infty)$ and denote the latter integral by $\mathcal{R}_p (t)$, we get 
\begin{equation}\label{eq:temp8}
\mathbb{E}\big[\Big(\ZZhf(2t, 0) e^{\frac{t}{12}}\big)^p\Big] = \frac{(-1)^n}{\Gamma(1-\alpha)}\int_0^{1} s^{-\alpha} \partial^n_s \Big(\mathbb{E}\big[\lt\big]\Big) ds + \mathcal{R}_p(t)
\end{equation}
It turns out $\mathcal{R}_p (t)$ is uniformly bounded in $t$ so we only need to focus on the first term on the right hand side \eqref{eq:temp8}. The laplace transform $\mathbb{E}\big[\lt\big]$ in the above integral admits an explicit formula in terms of the GOE point process given by the right hand side of \eqref{eq:exactformula}. Rewrite the expectation of products of the GOE point process into a Fredholm Pfaffian (Lemma \ref{lem:rains}), we get
\begin{equation*}
\mathbb{E}\big[\exp(-s \ZZhf(2t, 0)) e^{\frac{t}{12}}\big] = 1 + \sum_{L=1}^\infty \frac{1}{L!}\int_{\mathbb{R}^L} \text{Pf}\big[K(x_i, x_j)\big]_{i, j=1}^L \prod_{i=1}^L \phi_{s, t} (x_i) dx_i,
\end{equation*} 
where the Pfaffian kernel $K$ is $2\times 2$ matrix defined in Definition \ref{def:GOE} and the function $\phi_{s, t}$ is specified in \eqref{eq:phist}. Inserting the above expression to the first term on the right hand side of \eqref{eq:temp8}, bringing the infinite summation over $L$ outside the derivative over $s$ and integral from $0$ to $1$ (which will be justified), we get
\begin{equation*}
\mathbb{E}\big[\Big(\ZZhf(2t, 0) e^{\frac{t}{12}}\big)^p\Big] = \sum_{L=1}^\infty \frac{(-1)^n}{\Gamma(1-\alpha) L!} \int_0^1 s^{-\alpha} 
\int_{\mathbb{R}^L}\text{Pf}\big[K(x_i, x_j)\big]_{i, j=1}^L \partial_s^n \Big(\prod_{i=1}^L \phi_{s, t} (x_i)\Big) dx_1 \dots dx_L + \mathcal{R}_p (t).
\end{equation*} 
The next step is to decompose the infinite summation in the above display into $L=1$ and $L \geq 2$. In particular, we denote the first term in the summation by $\mathcal{A}_p (t)$ and the $L$-th term $(L \geq 2)$ by $\mathcal{B}_{p, L} (t)$, then
\begin{equation*}
\mathbb{E}\Big[\big(\ZZhf(2t, 0) e^{\frac{t}{12}}\big)^p\Big] = \mathcal{A}_p (t) + \sum_{L=2}^{\infty} \mathcal{B}_{p, L}(t) + \mathcal{R}_p (t)
\end{equation*}
We call $\mathcal{A}_p (t)$ the \emph{leading order term} which will be shown to hold the dominating $t \to \infty$ asymptotic. In particular, it equals an integral of the $(1,2)$ entry of the $2 \times 2$ matrix $K(x, x)$ denoted by $K_{12}(x, x)$. The second term on the right hand side of the above display is composed of \emph{the higher order terms}. Each $\mathcal{B}_{p, L}(t)$ is related to an integral of $L$-th correlation function $\text{Pf}\big[K(x_i, x_j)\big]_{i, j =1}^L$ of the GOE point process. In our proof, 
we show in Proposition \ref{prop:mainprop} and \ref{prop:minor} that for fixed $p > 0$, as $t \to \infty$, $\mathsf{1):}$ $\mathcal{A}_p (t)$ grows  asymptotically as $\exp(p^3 t/3)$.  $\mathsf{2):}$ $\sum_{L=2}^\infty \big|\mathcal{B}_{p, L} (t)\big|$ is upper bounded by $\exp\big((p^3 - \delta_p)t/3\big)$  for some $\delta_p > 0$ . This demonstrates Theorem \ref{thm:main} \ref{item:thm1}.
\bigskip
\\
So far, we have followed the idea in \cite{DT19}. The analysis of the leading order term $\mathcal{A}_p (t)$ in Proposition \ref{prop:mainprop} involves a steepest descent type analysis of the integral of $K_{12}(x, x)$. The harder problem is to control the higher order terms. In the situation of \cite{DT19}, the authors deal with a Fredholm determinant $\det(I + A) = 1 + \sum_{L=1}^\infty \text{Tr}(A^{\wedge L})$,\footnote{There is a misstatement in page 4 of \cite{DT19} where an extra $L!$ appears in the definition of Fredholm determinant.} where $A$ is a positive, trace class operator. \cite{DT19} upper bounds the higher order term $\text{Tr}(A^{\wedge L})$  by $\big(\text{Tr} (A)\big)^L/L! $. This can be understood by setting $\{\lambda_i\}_{i=1}^\infty$ to be the eigenvalues of $A$, since $\lambda_i$ are all non-negative, 
$$\text{Tr}(A^{\wedge L}) = \sum_{1 \leq i_1 < \dots < i_L} \lambda_{i_1} \dots \lambda_{i_L} \leq \frac{1}{L!}\big(\sum_{i=1}^\infty \lambda_i\big)^L = \frac{1}{L!}\big(\text{Tr} (A)\big)^L.$$  
Unfortunately, we are unaware of an analogue for such bound in the case of Fredholm Pfaffian. Instead of mimicking \cite{DT19}, we adopt a more direct approach. Using Hadamard's inequality and a determinantal analysis, we obtain two upper bounds of  $\text{Pf}\big[K(x_i, x_j)\big]_{i, j=1}^L$ (Proposition \ref{prop:pfbound}). These upper bounds will be applied to control various terms in $\sum_{L=2}^{\infty} \mathcal{B}_{p, L}(t)$ depending on whether $L$ is greater than a fixed threshold. We want to highlight that we did not pursue to get the sharp bound of the $p$-th moment of $\ZZhf(2t, 0)$ which holds uniformly for large $p$ and $t$, as shown in \cite[Theorem 1.1 (a)*]{DT19}. However, it is sufficient to apply our method to  obtain the  Lyapunov exponents and LDP for our problem as well as for \cite[Theorem 1.1]{DT19}.
\subsection{Previous results}\label{sec:previousresult}
Recently, there has been significant progress in understanding the Lyapunov exponents and tails of various Stochastic PDEs, see \cite[Section 1.2]{GL20} and reference therein. Here, we restrict our discussion to the scope of the KPZ equation and the SHE.
\subsubsection{Full line KPZ equation/SHE}
The KPZ equation was introduced in \cite{KPZ86} as a paradigmatic model for the random surface growth. It is a representative of the \emph{KPZ universality class} \cite{ACQ11, Cor12}, a collection of models sharing the universal scaling exponent and large time scaling behavior.  Recently, the upper/lower tail LDP of the KPZ equation receives plenty of attention from the mathematics and physics community. In fact, there are two regimes for the LDP of the KPZ equation, long time and short time (Freidlin-Wentzell regime). We will focus on discussing the long time regime and for the latter situation, see the physics literature \cite{KK07,KK09,KMS16, MKV16,LDMRS16}. 
\bigskip
\\
The upper tail of the KPZ equation is closely connected to the Lyapunov exponents of SHE. \cite{BC95} first computes the integer Lyapunov exponents of the SHE. However, due to an incorrect use of Skorokhod’s lemma, their result is only valid for the second moment. By analyzing the Brownian local time from Feynman-Kac representation of the SHE, \cite{Chen15} obtained the integer Lyapunov exponents for the SHE under deterministic bounded initial data.   For the SHE under narrow wedge initial condition, the integer moment of the solution 
admits a contour integral formula \cite{BC14, Gho18}. By a residue calculus, \cite{CG20a} obtains the integer Lyapunov exponents, from which they obtain a bound for the upper tail of the KPZ equation, showing the correct exponent $3/2$. \cite{DT19} improves their result by identifying all positive real Lyapunov exponents of the SHE. As a consequence, they obtain the upper tail LDP of the KPZ equation with rate function $\frac{4}{3} s^{\frac{3}{2}}$. 
Recently,  \cite{GL20} is capable of computing all positive Lyapunov exponents for the SHE starting with a class of general (including random) initial data and obtain the corresponding upper tail LDP of the KPZ equation.
\bigskip
\\
Unlike the upper tail, the lower tail of the KPZ equation does not have a strong connection to the moment of SHE. For the narrow wedge initial condition, via a delicate analysis of the exact formula in \cite{BG16}, \cite{CG20b} derives a tight bound which detects the crossover of the tail exponent  from $3$ to $5/2$ depending on the depth of the tail, which was first observed in the physics work \cite{SMP17}. The LDP for the lower tail of the KPZ equation is obtained by \cite{Tsai18, CC19}. For the KPZ equation with general initial data, \cite{CG20b} obtains an upper bound for the lower tail probability. Besides that, very few things are known at present. 
\subsubsection{Half-line KPZ equation/SHE}
Compared with the knowledge for full-line equation, smaller amount of results are known for the half-line KPZ equation/SHE. 
\cite{CS18} proves that on a closed interval or a half line, the open ASEP weakly converges to the KPZ equation with Neumann boundary parameter $A \geq 0$. Such convergence was extended later by \cite{Par19} to all $A \in \mathbb{R}$. \cite{BBCW18, Par19} obtain the Laplace transform formula for the half-line SHE under narrow wedge initial condition when $A = -\frac{1}{2}$, which helps to capture the Tracy-Widom GOE fluctuation of the KPZ equation. 
As discussed before, there is a conjectured Gaussian-GOE-GSE phase transition the for half-line KPZ equation, which is only proved at the critical parameter $A = -\frac{1}{2}$. Recently, there are progress identifying new limiting distribution for the half-space KPZ equation under stationary initial data \cite{BKD20}. Such distribution is believed to be universal and arises in other half-space model starting from stationary initial data \cite{BFO20}.
\bigskip
\\
Regarding the tail of half-space KPZ equation, let us focus on $A = -\frac{1}{2}$ and narrow wedge initial condition. Results for other boundary parameters and initial conditions are fairly untouched for now. Under the aforementioned boundary condition, a tight estimate of the lower tail was obtained in \cite{Kim19}, which detects the similar crossover of the tail exponent that appears in the full-line situation. The LDP for the lower tail was obtained by \cite{Tsai18}. Few things were rigorously proved for the upper tail aside from the current work. \cite{BBC16} (also see \cite{BBC20}) obtains a moment formula of the half-line SHE by solving the delta-Bose gas (the result is not rigorous, since the uniqueness of the solution to the delta-Bose gas is unknown). It is also unclear whether one can extract the integer Lyapunov exponents  for the half-line SHE (thus obtaining tail bounds of the half-line KPZ equation) from a similar residue calculus of the integral formula as carried out in \cite{CG20a}, due to the extra complexity.
\bigskip
\\
On a different aspect, it is worth to mention the works of \cite{KLD18, MV18} in which the authors consider the LDP for the half-line KPZ equation in short time. The same exponent $3/2$ in the rate function is obtained therein. In addition, \cite{KLD18} predicted  the upper tail exponent to be $\frac{2}{3} x^{\frac{3}{2}}$ when $A = -\frac{1}{2}, 0$ and $\frac{4}{3} x^{\frac{3}{2}}$ when $A = +\infty$. For the future work, it is appealing to prove a LDP for the upper tail for general boundary parameter $A$ and see how the LDP rate function changes when $A$ belongs to the Gaussian/GSE regime.
\bigskip
\\
\textbf{Outline.} The rest of the paper is organized as follows. In section \ref{sec:overview}, we give an overview for the proof of the main theorem and provide more details for what is discussed in Section \ref{sec:mainresult}. In particular, we transform our problem into proving Proposition \ref{prop:mainprop} and \ref{prop:minor}. Section \ref{sec:Ap} was devoted to prove Proposition \ref{prop:mainprop}. In Section \ref{sec:prepare}, we provide two different upper bound for the pfaffian $\text{Pf}\big[K(x_i, x_j)\big]_{i, j =1}^L$ which is crucial to the proof of Proposition \ref{prop:minor}. We also justify Lemma \ref{lem:intdevinterchange} in that section. Section \ref{sec:Bp} completes the proof of Proposition \ref{prop:minor}.
\bigskip
\\
\textbf{Acknowledgment.} The author thanks Guillaume Barraquand, Ivan Corwin, Sayan Das and Li-Cheng Tsai for helpful discussions. The author was partially supported by the Fernholz Foundation's ``Summer Minerva Fellow" program and also received summer support from Ivan Corwin's
NSF grant DMS-1811143, DMS-1664650.
\section{Proof of Theorem \ref{thm:main}: A detailed overview}\label{sec:overview}
 In this section, we explain with more details how we arrive at proving Theorem \ref{thm:main}. We begin with the definition of GOE point process that was mentioned in the introduction. As formulated in \cite[Section 4.2.1]{AGZ10}, a point process on $\mathbb{R}$ is a random point configuration $\mathcal{X}$. The $L$-th correlation function $\rho_L$ w.r.t. the measure $\mu$ associated to $\mathcal{X}$ is defined in the way that for arbitrary distinct Borel sets  $B_1, \dots, B_L$,
\begin{equation*}
\int_{B_1 \times \dots \times B_L} \rho_L (x_1, \dots, x_L) d\mu^{\otimes L} = \mathbb{E}\Big[\#\big\{(x_1, \dots, x_L), \text{ such that } x_i \in \mathcal{X} \cap B_i, i = 1, \dots, L \big\}\Big].
\end{equation*} 
We say a point process is a Pfaffian, if there exists a  matrix kernel $K: \mathbb{R} \times \mathbb{R} \to \mathbb{C}^{2 \times 2}$ such that the correlation function $\rho_L (x_1, \dots, x_L) = \text{Pf}\big[K(x_i, x_j)\big]_{i, j = 1}^L$ for arbitrary $L \in \mathbb{Z}_{\geq 1}.$   
\begin{defin}[GOE point process]\label{def:GOE}
We say 
$\mathcal{X} = \{\mathbf{a}_1 > \mathbf{a}_2 > \cdots\} $ is the GOE-Airy point process, if it is a Pfaffian point process on $\mathbb{R}$ with kernel 
\begin{equation*}
K(x, y) = \begin{bmatrix}
K_{11} (x, y), & K_{12} (x, y) \\
K_{21} (x, y), & K_{22} (x, y)
\end{bmatrix}
\end{equation*}
with the entries 
\begin{align}
\label{eq:thmk11}
K_{11} (x, y) &= \int_0^{\infty} \emph{Ai}(x+\lambda) \emph{Ai}'(y + \lambda)  - \emph{Ai}(y+\lambda) \emph{Ai}'(x+\lambda) d\lambda,\\
\label{eq:thmk12}
K_{12} (x, y) &= -K_{21} (y, x) = \frac{1}{2} \int_0^{\infty} \emph{Ai}(x+\lambda) \emph{Ai} (y+\lambda) d\lambda + \frac{1}{2} \emph{Ai}(x) \int_{-\infty}^y \emph{Ai}(\lambda) d\lambda ,\\
\notag
K_{22} (x, y) &= \frac{1}{4} \int_0^\infty  \Big(\int_\lambda^\infty  \emph{Ai}(y+\mu) d\mu\Big) \emph{Ai}(x + \lambda) d\lambda 
- \frac{1}{4} \int_0^{\infty}  \Big(\int_{\lambda}^{\infty}  \emph{Ai}(x+\mu) d \mu\Big) \emph{Ai}(y+ \lambda) d\lambda\\ 
\label{eq:thmk22}
&\quad - \frac{1}{4} \int_0^\infty \emph{Ai}(x+\lambda) d\lambda + \frac{1}{4} \int_0^{\infty} \emph{Ai}(y+\lambda) d\lambda - \frac{\emph{sgn}(x-y)}{4}.
\end{align} 
Here, $\emph{sgn}(x)$ is defined as the sign function $\mathbf{1}_{\{x > 0\}} - \mathbf{1}_{\{x < 0\}}$. Furthermore, we set $K_{21}(x, y) = -K_{12} (y, x)$.
\end{defin}
\begin{remark}
Note that our expression of the kernel \eqref{eq:thmk11}, \eqref{eq:thmk12}, \eqref{eq:thmk22} is different from that in Eq. (6.1a), (6.1b), (6.1c) of \cite{BBCW18}. However, they are demonstrated to be the same, see (2.9) and (6.17) of \cite{Ferrari04} or \cite[Lemma 2.6]{BBCS18}.
\end{remark}
 It turns out that the right hand side of \eqref{eq:exactformula} can be rewritten as a \emph{Fredholm Pfaffian},  which has been first defined in \cite{Rains00}. We reproduce the definition of the Fredholm Pfaffian from \cite[
Definition 2.3]{BBCS18}.
\begin{defin}[Fredholm Pfaffian]\label{def:fredholmpf}
Let $K(x, y)$ be an asymmetric $2 \times 2$ matrix and $\mu$ to be a measure on $\mathbb{R}$ and  $f: \R \to \mathbb{C}$ be a measurable function. We define the Fredholm Pfaffian by the series expansion 
\begin{equation*}
\emph{Pf}\big[J + K\big]_{L^2(\mathbb{R}, f\mu)}  = \sum_{L = 0}^\infty \frac{1}{L!} \int_{\RL} \emph{Pf}\big[K(x_i, x_j)\big]_{i, j=1}^L \Big(\prod_{i=1}^L f(x_i)\Big) d\mu^{\otimes L} (x_1, \dots, x_L),
\end{equation*}
where 
\begin{equation*}
J(x, y) = \mathbf{1}_{\{x = y\}}
\begin{bmatrix}
0 & 1 \\
-1 & 0
\end{bmatrix}.
\end{equation*}
\end{defin} 
\begin{lemma}[Pfaffian point process and Fredholm Pfaffian \cite{Rains00}]\label{lem:rains}
Let $\mathbf{a_1} > \mathbf{a}_2 > \dots $ be a Pfaffian point process with kernel $K$ and  $f: \R \to \mathbb{C}$ be a measurable function. We have 
\begin{equation}\label{eq:temp5}
\mathbb{E}\Big[\prod_{i = 1}^\infty \big(1 + f(\mathbf{a}_i)\big)\Big] = \emph{Pf}\big[J + K\big]_{L^2(\mathbb{R}, f\mu)},
\end{equation}
as long as both sides of the above equation converge absolutely.
\end{lemma}
 We take $\mathbf{a}_1 > \mathbf{a_2} > \dots$ to be the GOE point process defined in Definition \ref{def:GOE} and the function $f$ in the above lemma to be 
\begin{equation}\label{eq:phist}
\phi_{s, t} (x) := \frac{1}{\sqrt{1 + 4s \exp(t^{1/3} x)}} - 1,
\end{equation}
Theorem 7.6 of \cite{BBCW18} has already justified the convergence of both sides of \eqref{eq:temp5}. As a result,
\begin{align*} \mathbb{E}\Big[\prod_{i=1}^{\infty} \frac{1}{\sqrt{1 + 4s \exp(t^{\frac{1}{3}} \mathbf{a}_i)}}\Big]
= \sum_{L = 0}^{\infty} \frac{1}{L!} \int_{\RL} \pfll \prod_{i=1}^L \phi_{s, t}(x_i) dx_i.
\end{align*}
By Theorem \ref{thm:exactformula}, the left hand side in the above display equals the Laplace transform of $\ZZhf(2t, 0)\exp(\frac{t}{12})$, thus
\begin{equation}
\label{eq:pfaffianexpansion}
\mathbb{E}\Big[\lt\Big] = \sum_{L = 0}^{\infty} \frac{1}{L!} \int_{\RL} \pfll \prod_{i=1}^L \phi_{s, t}(x_i) dx_i
\end{equation}
To prove Theorem \ref{thm:main}, the next step is to link the Laplace transform of $\ZZhf(2t, 0)\exp(\frac{t}{12}) $ with its fractional moment. The following lemma was stated as \cite[Lemma 1.2]{DT19}, which can be verified via Fubini's theorem. 
\begin{lemma}\label{lem:momentlaplace}
For arbitrary non-negative random variable $X$, $0 \leq \alpha < 1$ and $n \in \Zp$,
\begin{align*}
\mathbb{E}\Big[X^{n-1+\alpha}\Big] &= \frac{(-1)^n}{\Gamma(1-\alpha)}  \int_0^{\infty} s^{-\alpha} \partial_s^n \mathbb{E}\Big[e^{-sX}\Big] ds.
\end{align*}
As a convention, we use $\partial_s^n \mathbb{E}\big[e^{-sX}\big]$ to denote the $n$-th derivative of $\mathbb{E}\big[e^{-sX}\big]$ with respect to $s$.
\end{lemma} 
 For fixed $p > 0$, we set $n = \lfloor p \rfloor + 1$ and $\alpha = p - n + 1$. It is clear that  $n \in \Zp$, $\alpha \in [0, 1)$. Applying Lemma \ref{lem:momentlaplace} with $X = \ZZhf(2t, 0) \exp\big(\frac{t}{12}\big)$ (note that $\ZZhf(2t, 0)$ is almost surely positive), we find that 
\begin{align*}
\mathbb{E}\Big[\ZZhf(2t, 0)^{p} e^{\frac{pt}{12}}\Big] = \frac{(-1)^n}{\Gamma(1-\alpha)}  \int_0^{\infty} s^{-\alpha} \partial_s^n \mathbb{E}\Big[e^{-s(\ZZhf(2t, 0) + \frac{t}{12})}\Big] ds.
\end{align*}
Splitting the interval of integration into $[0, 1]$ and $[1, \infty)$ yields
\begin{align}
\label{eq:temp4}
\mathbb{E}\Big[\ZZhf(2t, 0)^{p} e^{\frac{pt}{12}}\Big] = 
\frac{(-1)^n}{\Gamma(1-\alpha)}  \int_0^{1} s^{-\alpha} \partial_s^n \mathbb{E}\Big[e^{-s(\ZZhf(2t, 0) + \frac{t}{12})}\Big] ds + \mathcal{R}_p(t),
\end{align}
where $\mathcal{R}_{p}(t) := \frac{(-1)^n}{\Gamma(1-\alpha)} \int_1^{\infty} s^{-\alpha} \partial_s^n \mathbb{E}\big[e^{-s(\ZZhf(2t, 0) + \frac{t}{12})}\big] ds.$ 
\begin{lemma}\label{lem:Rp}
For fixed $p > 0$, $|\mathcal{R}_{p}(t)|$ is uniformly bounded by a constant  for every $t > 0$. 
\end{lemma}
\begin{proof}
Since $\mathcal{R}_p (\tt) = \frac{(-1)^n}{\Gamma(1-\alpha)}\int_1^{\infty} s^{-\alpha} \mathbb{E}\big[e^{-sX} X^n\big] ds$
with $X = \ZZhf(2\tt, 0) \exp\big(\frac{t}{12}\big)$. Note that  $$\mathbb{E}\Big[e^{-sX} X^n\Big] \leq \sup_{x \geq 0}\big( e^{-sx} x^n\big) = s^{-n} n^n e^{-n}.$$ 
Replacing $\mathbb{E}\big[e^{-sX } X^n\big]$ with this upper bound inside the integral yields
\begin{equation*}
0 \leq (-1)^n \mathcal{R}_{p} (t) \leq \frac{1}{\Gamma(1-\alpha)} \int_1^{\infty} s^{-n-\alpha} n^n e^{-n} ds = \frac{n^n e^{-n}}{\Gamma(1-\alpha)(n+\alpha)}.
\end{equation*}
Since $\alpha$ and $n$ are determined by $p$, so the right hand side is a constant that only depends on $p$, this completes our proof.
\end{proof}
By \eqref{eq:pfaffianexpansion}, we see that the first term on the RHS of \eqref{eq:temp4} can be written as 
\begin{align}\label{eq:dis}
\int_0^{1} s^{-\alpha} \partial_s^n \mathbb{E}\big[e^{-s(\ZZhf(2t, 0) + \frac{t}{12})}\big] ds = \int_0^{1} s^{-\alpha} \partial_s^n \Big(\sum_{L = 1}^\infty \frac{1}{L!}\int_{\RL} \pfll  \prod_{i = 1}^L  \phi_{s, t}(x_i) dx_i\Big) ds.
\end{align}
Note that we throw out the $L = 0$ term in the summation since it is always $1$ and has $s$-derivative to be $0$.
It turns out that we can interchange the order of derivative, integration and summation for the right hand side of the above display, for which we formulate as a lemma. The proof of it is deferred to Section \ref{sec:dis}.
\begin{lemma}\label{lem:intdevinterchange}
We have 
\begin{align}
\notag
&\int_0^{1} s^{-\alpha} \partial_s^n \Big(\sum_{L = 1}^\infty \frac{1}{L!}\int_{\RL} \emph{Pf} \big[K(x_i, x_j)\big]_{i, j = 1}^L \prod_{i = 1}^L  \phi_{s, t}(x_i) dx_i\Big) ds\\
\label{eq:laplace1}
&= \sum_{L = 1}^\infty \frac{1}{L!}\int_0^{1} s^{-\alpha} ds   \int_{\RL} \emph{Pf} \big[K(x_i, x_j)\big]_{i, j = 1}^L  \partial_s^n\Big(\prod_{i = 1}^L  \phi_{s, t}(x_i)\Big) dx_1 \dots dx_L.
\end{align}
Consequently, it follows from \eqref{eq:dis} and the above display  that 
\begin{equation}\label{eq:laplace2}
\int_0^{1} s^{-\alpha} \partial_s^n \mathbb{E}\Big[\lt\Big] ds = \sum_{L = 1}^\infty \frac{1}{L!}\int_0^{1} s^{-\alpha} ds   \int_{\RL} \emph{Pf} \big[K(x_i, x_j)\big]_{i, j = 1}^L  \partial_s^n \Big(\prod_{i = 1}^L  \phi_{s, t}(x_i)\Big) dx_1 \dots dx_L.
\end{equation}
\end{lemma}
We set the first term in the right hand side summation of \eqref{eq:laplace2} as $\AA(t)$ and the higher order terms as $\BBpL(t)$ ($L \geq 2$), i.e. 
\begin{align}
\label{eq:ap}
\AA (t) 
&= \frac{(-1)^n}{\Gamma (1-\alpha)} \int_0^{1} s^{-\alpha}  \int_{\R}  K_{12} (x, x) \big(\partial_s^n \phi_{s, t}(x)\big) dx\\
\label{eq:bpL}
\mathcal{B}_{p, L} (t) &= \frac{(-1)^n}{\Gamma (1 - \alpha) L!}  \int_0^{1} s^{-\alpha}    \int_{\RL} \pfll  \partial_s^n \Big(\prod_{i = 1}^L \phi_{s, t}(x_i)\Big) dx_1 \dots dx_L, \quad L \geq 2.
\end{align}
Under this notation, the left hand side of \eqref{eq:laplace2} equals $\AA(t) + \sum_{p=2}^\infty \BBpL(t)$. Referring to \eqref{eq:temp4}, we obtain 
\begin{equation}\label{eq:decomposition}
\mathbb{E}\Big[\ZZhf(2t, 0)^{p} e^{\frac{pt}{12}}\Big]  = \mathcal{A}_{p} (t) + \sum_{L=2}^\infty \mathcal{B}_{p, L} (t) + \mathcal{R}_p (t).
\end{equation}
We want to show that the logarithm of the left hand side in the above display, after divided by $t$ and letting $t \to \infty$, converges to $p^3/3$. By Lemma \ref{lem:Rp}, $|\mathcal{R}_p(t)|$ is uniformly upper bounded by a constant for all $t > 0$. Therefore, to prove Theorem \ref{thm:main}, it suffices to demonstrate that the following facts for $\AA(t)$ and $\sum_{L=2}^\infty |\BBpL(t)|$.
\begin{prop}\label{prop:mainprop}
For fixed $p \in \mathbb{R}_{> 0}$,
$\lim_{t \to \infty} \frac{1}{t} \log \mathcal{A}_{p} (t) = e^{\frac{p^3}{3}}$.
\end{prop}
\begin{prop}\label{prop:minor}
For fixed $p \in \mathbb{R}_{> 0}$,
$\limsup_{t \to \infty} \frac{1}{t} \log\Big(\sum_{L=2}^\infty |\mathcal{B}_{p, L} (\tt)|\Big) \leq e^{\frac{p^3}{3} - \delta_p}$ with $\dd = \min(\frac{2}{3}, \frac{p^3}{4})$.
\end{prop}
\noindent We will prove the two propositions in Section \ref{sec:Ap} and Section \ref{sec:Bp}. Let us first conclude the proof of Theorem  \ref{thm:main}.
\begin{proof}[Proof of Theorem \ref{thm:main}]
For part \ref{item:thm1}, from Proposition \ref{prop:mainprop} and \ref{prop:minor}, we know that $\mathcal{A}_{p} (t)$ grows exponentially faster than $\sum_{L=2}^\infty |\mathcal{B}_{p, L} (t)|$ as $t \to \infty$. Along with the fact that $|\mathcal{R}_{p}(t)|$ is upper bounded by a constant for all $t$,  there exists  $t_0 >  0$ such that for all $t > t_0$, $$\sum_{L=2}^\infty |\mathcal{B}_{p, L} (\tt)| + \big|\mathcal{R}_{p}(t)\big| \leq \frac{1}{2} \mathcal{A}_{p} (t).$$ 
Referring to the decomposition \eqref{eq:decomposition} and using triangle inequality, we see that for $t > t_0$,
\begin{equation*}
\log\Big(\frac{1}{2}\mathcal{A}_{p} (t)\Big) \leq \log \mathbb{E}\Big[\ZZhf(2t, 0)^p\Big] \leq \log\Big( \frac{3}{2}\mathcal{A}_{p} (t)\Big).
\end{equation*}
Dividing very term in the above display by $t$, Theorem \ref{thm:main} \ref{item:thm1} follows easily from Proposition \ref{prop:mainprop} as we take $t \to \infty$. Since we know that $t^{-1} \lim_{t \to \infty} \log \mathbb{E}\big[\ZZhf(2t, 0)^p \exp(\frac{pt}{12})\big] = p^3/3$. Applying \cite[Proposition 1.12]{GL20} with $h(p) = \frac{p^3}{3}$, we obtain the upper tail LDP with rate function to be $-\sup_{p > 0} (ps - p^3/3) = \frac{2}{3} s^{\frac{3}{2}}$,  thus  we  obtain Theorem \ref{thm:main} \ref{item:thm2}.
\end{proof}

\section{Asymptotic of $\mathcal{A}_{p}(t)$: Proof of Proposition \ref{prop:mainprop}}
\label{sec:Ap}
 
In this section, we prove Proposition \ref{prop:mainprop}. One crucial step is Lemma \ref{lem:k12laplace}, whose proof relies on  Lemma \ref{lem:k12bound} and a steepest descent type analysis. Throughout the rest of the paper, we use $C, C_1, C_2$ to denote a constant, which may vary from line to line. We might not generally specify when irrelevant terms are being absorbed into the constants. We might also write $C(a), C(a, b)$ when we want to specify which parameters the constant depends on.
\begin{lemma}\label{lem:gamma}
Denote $\Be(u, v)$ to be the beta function $\int_0^1 x^{u-1} (1-x)^{v-1} dx$. For $\gamma > 0, \alpha < 1$ and $\alpha + \beta > 1$,
\begin{equation*}
\int_0^\infty \frac{s^{-\alpha}}{(1 + \gamma s)^{\beta}} ds=  \gamma^{\alpha - 1} \Be(1 - \alpha, \beta+\alpha - 1)
\end{equation*}
\end{lemma}
\begin{proof}
Via a change of variable $s = \frac{t}{\gamma(1-t)}$, 
we get 
\begin{equation*}
\int_0^\infty \frac{s^{-\alpha} ds}{(1+\gamma s)^\beta} = \gamma^{\alpha-1} \int_0^1 t^{-\alpha} (1-t)^{\alpha + \beta -2} dt =  \gamma^{\alpha - 1} \Be(1 - \alpha, \beta+\alpha - 1). \qedhere
\end{equation*}
\end{proof}

\begin{lemma}\label{lem:k12laplace}
For fixed $p, t_0 > 0$, there exists constant $\C = C(p, t_0)$ such that for all $\tt > \tt_0$,
\begin{equation*}
\frac{1}{C} t^{-\frac{2}{3}} e^{\frac{1}{3} p^3 t} \leq\int_0^\infty K_{12} (t^{\frac{2}{3}} x, t^{\frac{2}{3}} x) e^{p t x} dx \leq C t^{-\frac{2}{3}} e^{\frac{1}{3} p^3 t}.
\end{equation*}
\end{lemma}
\begin{proof}
Throughout the proof we write $C = C(p, t_0)$ and denote by $\UU_p (x) = -\frac{2}{3} x^{\frac{3}{2}} + px$. Using the inequality in Lemma \ref{lem:k12bound} \ref{item:k12 xpositive}, 
\begin{align*}
\frac{1}{C}\int_0^{\infty} \frac{e^{t \UU_p (x)}}{(1+t^{\frac{2}{3}} x)^{\frac{1}{4}}} dx \leq \int_0^\infty K_{12} (t^{\frac{2}{3}} x, t^{\frac{2}{3}} x) e^{p t x} dx \leq C \int_0^{\infty} \frac{e^{t \UU_p (x)}}{(1+t^{\frac{2}{3}} x)^{\frac{1}{4}}} dx
\end{align*}
The proof is completed if we can show there exists a constant $C$ such that for all $t > t_0$,
\begin{equation}\label{eq:k12needtoshow}
\frac{1}{\C} t^{-\frac{2}{3}} e^{\frac{p^3 \tt}{3}} \leq \int_0^{\infty} \frac{e^{t \UU_p (x)}}{(1+t^{\frac{2}{3}} x)^{\frac{1}{4}}} dx \leq \C t^{-\frac{2}{3}}e^{\frac{p^3 t}{3}}
\end{equation} 
An elementary calculus tells that the maximum of $\UU_p (x) = -\frac{2}{3} x^{\frac{3}{2}} + px$ on $[0, \infty)$ is reached at $x = p$, with $U_p (p) = \frac{1}{3} p^3$. So it is natural to expect that the main contribution of the integral in the above display comes around a small region around $x = p$.  
Having this intuition in mind, we let $\qq = \frac{p}{4}$ decompose
\begin{equation}\label{eq:laplacebreak}
\int_0^{\infty} \frac{e^{t \UU_p (x)}}{(1+t^{\frac{2}{3}} x)^{\frac{1}{4}}} dx = \Big(\int_{[(p-\qq)^2, (p+\qq)^2]} +   \int_{\Rp \backslash [(p-\qq)^2, (p+\qq)^2]}\Big) \frac{e^{t \UU_p (x)}}{(1+t^{\frac{2}{3}}x)^{\frac{1}{4}}} dx = \KK_1 + \KK_2.
\end{equation}
It suffices to analyze $\KK_1$ and $\KK_2$ respectively. For $\KK_1$, we make a change of variable $x = (p + r)^2$. Noting that $\UU_p \big((p+r)^2\big) = \frac{p^3}{3} - (\frac{2}{3}r+p) r^2$, we get
\begin{equation}\label{eq:k1}
\KK_1 = \int_{-\qq}^\qq \frac{2(p+r) e^{t \UU_p ((p+r)^2)} }{(1+t^{\frac{2}{3}} (p+r)^2)^{\frac{1}{4}}} dr =  e^{\frac{p^3 t}{3}} \int_{-\qq}^{\qq} \frac{2(p + r) e^{-t (\frac{2}{3} r + p) r^2}}{ (1 + t^{\frac{2}{3}} (p+r)^2)^{\frac{1}{4}}} dr
\end{equation}
Recall that $\qq = \frac{p}{4}$, so there exists a constant $C = C(p, t_0)$ such that for all $r \in [-\frac{p}{4}, \frac{p}{4}]$ and $t > t_0$, 
\begin{equation}\label{eq:temp7}
\frac{e^{-Ct r^2}}{C t^{\frac{1}{6}}} \leq \frac{2(p+r) e^{-t (\frac{2}{3} r + p) r^2}}{(1 + t^{\frac{2}{3}} (p+r)^2)^{\frac{1}{4}}} \leq \frac{C e^{-\frac{1}{C} t r^2}}{t^{\frac{1}{6}}}.
\end{equation}
By a change of variable $r \to t^{-\frac{1}{2}} r$, there exists constant $C_1$ such that for $t > t_0$
\begin{equation*}
C_1^{-1} t^{-\frac{2}{3}} \leq \int_{-\qq}^\qq \frac{ e^{-C t r^2}}{C t^{\frac{1}{6}}} \leq \int_{-\qq}^\qq \frac{C e^{-\frac{1}{C} t r^2}}{t^{\frac{1}{6}}} \leq C_1 t^{-\frac{2}{3}}
\end{equation*}
Integrating the terms in \eqref{eq:temp7} from $-\qq$ to $\qq$ and utilizing the  displayed inequality above and \eqref{eq:k1}, we conclude that  
$\frac{1}{C} t^{-\frac{2}{3}} e^{\frac{p^3 t}{3}} \leq \KK_1 \leq C t^{-\frac{2}{3}} e^{\frac{p^3 t}{3}}$ for $t> t_0$.
\bigskip
\\
For $\KK_2$, by a change of variable $x = r^2$ and noting $U_p (r^2) = \frac{p^3}{3}-(r-p)^2 (\frac{2}{3}r + \frac{1}{3}p)$, we have 
\begin{equation}
\KK_2  
= e^{\frac{p^3 \tt}{3}} \int_{\Rp \backslash [ p-\qq, p+\qq]} \frac{e^{\tt (r-p)^2 (-\frac{2}{3} r - \frac{1}{3} p)}}{(1 + \tt^{\frac{2}{3}} r^2)^{\frac{1}{4}}} dr 
\leq e^{\frac{\tt( p^3 -p\qq^2)}{3}} \int_{\Rp \backslash [ p-\qq, p+\qq]} \frac{e^{-\frac{2}{3} t r(r-p)^2 }}{(1+t^{\frac{2}{3}} r^2)^{\frac{1}{4}}} dr
\end{equation}
The inequality in the above display follows from noticing $\frac{1}{3}(r-p)^2 p \geq \frac{p\qq^2}{3}$ when $r \notin [p-\qq, p+\qq]$. For the integral on the right hand side of the above display, we find that  $\int_{\Rp \backslash [ p-\qq, p+\qq]} \frac{e^{-\frac{2}{3} \tt r(r-p)^2 }}{(1+\tt^{\frac{2}{3}} r^2)^{\frac{1}{4}}} dr \leq \int_{\Rp} e^{-\frac{2}{3} \tt \qq^2 r} dr = \frac{3}{2\qq^2 \tt}$. Since we assume $t \geq t_0$, by taking $C = \frac{3}{2\qq^2 t_0}$, we know that 
$$0 \leq \KK_2 \leq \frac{3}{2\qq^2 t} e^{\frac{\tt( p^3 -p\qq^2)}{3}} \leq C e^{\frac{\tt( p^3 -p\qq^2)}{3}}.$$  Combining this with \eqref{eq:k1} and recall from \eqref{eq:laplacebreak} that $\int_0^{\infty} \frac{e^{t \UU_p (x)}}{(1+t^{\frac{2}{3}} x)^{\frac{1}{4}}} dx = \KK_1 + \KK_2$, we see that $\KK_1$ is the dominating term. This completes the proof of \eqref{eq:k12needtoshow}.
\end{proof}
 We are now ready to prove Proposition \ref{prop:mainprop}.
\begin{proof}[Proof of Proposition \ref{prop:mainprop}]
Recall  from \eqref{eq:phist} that $\phi_{s, t}(x) = \frac{1}{\sqrt{1 + 4s \exp(t^{1/3} x)}} - 1$, so 
$$\partial_s^n \phi_{s, t}(x) = (-2)^n (2n-1)!! (1 + 4s\exp(t^{\frac{1}{3}}x))^{-\frac{2n+1}{2}}.$$
By Fubini's theorem,  we switch the order of integration on the right hand side of \eqref{eq:ap}, hence
\begin{align*}
\mathcal{A}_{p} (t) 
&= \frac{2^n (2n-1)!! }{ \Gamma(1-\alpha)} \int_{\R} K_{12}(x, x) e^{n t^{\frac{1}{3}} x} dx \int_0^1 \frac{s^{-\alpha}}{(1 + 4s\exp(t^{\frac{1}{3}}x))^{\frac{2n+1}{2}}} ds,
\end{align*}
Writing the integral w.r.t $s$ in the above display as $\int_0^1  = \int_0^{\infty} - \int_1^{\infty}$, we get $\mathcal{A}_{p} (t) = \AAp (t) - \AApp (t),$ where 
\begin{align}
\label{eq:AAp}
\AAp (t) &= \frac{2^n (2n-1)!!}{ \Gamma(1-\alpha)} \int_{\R} K_{12}(x, x) e^{n t^{\frac{1}{3}} x} dx \int_0^{\infty} \frac{s^{-\alpha}}{(1 + 4s\exp(t^{\frac{1}{3}}x))^{\frac{2n+1}{2}}} ds, \\
\label{eq:alphadecompose}
\AApp(t) &=\frac{2^n (2n-1)!!}{ \Gamma(1-\alpha)} \int_{\R} K_{12}(x, x) e^{n t^{\frac{1}{3}} x} dx \int_1^\infty \frac{s^{-\alpha}}{(1 + 4s\exp(t^{\frac{1}{3}}x))^{\frac{2n+1}{2}}} ds.
\end{align}
To conclude our proof of Proposition \ref{prop:mainprop}, it suffices to show the following propositions. 
\begin{prop}\label{prop:AAp}
For fixed $p, t_0 > 0$ there exists $C = C(p, t_0)$ such that for all $t > t_0$, $\frac{1}{C} e^{\frac{1}{3}p^3 t}  \leq \AAp (t) \leq C e^{\frac{1}{3}p^3 t}$. 
\end{prop}
\begin{prop}\label{prop:AApp}
For fixed $p, t_0 > 0$, there exists $C = C(p, t_0)$ such that for all $t > t_0$,  $|\AApp (t)| \leq C$. 
\end{prop}
\noindent Let us first complete our proof of Proposition \ref{prop:mainprop} using Proposition \ref{prop:AAp} and \ref{prop:AApp}. Recall $\mathcal{A}_p (t) = \AAp(t) - \AApp (t)$. With the help of these lemmas, it is clear that $\AAp (t)$ is the dominating term for large enough $t$. Hence, $$\lim_{t \to \infty} \frac{1}{t} \log \AA(t) = \lim_{t \to \infty} t^{-1}\log \AAp(t) = \frac{p^3}{3}.$$ 
This completes our proof of  Proposition \ref{prop:mainprop}.
\end{proof} 
 For the rest of this section, we prove Proposition \ref{prop:AAp} and  \ref{prop:AApp} respectively.
\begin{proof}[Proof of Proposition \ref{prop:AAp}]
Applying Lemma \ref{lem:gamma} to  the second integral on the right hand side of  \eqref{eq:AAp} (with $\beta = \frac{2n + 1}{2}$ and $\gamma = 4\exp(t^{\frac{1}{3}} x)$), we see that (recall $p = n-1+\alpha$)
$$
\AAp (t)  = c_p \int_{-\infty}^{\infty} K_{12} (x, x) \exp\big(p t^{\frac{1}{3}} x\big) dx.
$$
where $c_p$ is a constant that equals $\frac{2^n (2n-1)!!\, 4^{\alpha-1}}{ \Gamma(1-\alpha)} \Be\big(1-\alpha, \frac{2n-1}{2} + \alpha\big)$. We will not use this explicit expression of $c_p$ and later we just write it as a generic constant $C$. By a change of variable $x \to t^{\frac{2}{3}} x$, we have  
$
\AAp(t) = C t^{\frac{2}{3}} \int_{-\infty}^{\infty} K_{12} (t^{2/3} x, t^{2/3} x) \exp(p t x) dx
$.
Decompose the integral region into $(-\infty, 0) \cup [0, \infty)$, we obtain
\begin{align}\label{eq:App}
\AAp(t) = C t^{\frac{2}{3}} &\bigg( \int_0^\infty K_{12} (t^{\frac{2}{3}} x, t^{\frac{2}{3}} x) e^{p t x} dx +  \int_{-\infty}^0 K_{12}(t^{\frac{2}{3}} x, t^{\frac{2}{3}} x) e^{p t x} dx\bigg)
\end{align}
For the first integral on the right hand side of the \eqref{eq:App}, referring to Lemma \ref{lem:k12laplace}, we have  
\begin{align}\label{eq:temp23}
\frac{1}{C_1} t^{-\frac{2}{3}} e^{\frac{p^3 t}{3}} \leq \int_0^\infty K_{12} (t^{\frac{2}{3}} x, t^{\frac{2}{3}} x) e^{p t x} dx \leq C_1 t^{-\frac{2}{3}} e^{\frac{p^3 t}{3}}.
\end{align}
For the second integral on the right hand side of \eqref{eq:App}, we apply Lemma \ref{lem:k12bound} (ii) and get $t > t_0$
\begin{align}\label{eq:temp24}
0\leq \int_{-\infty}^{0} K_{12}(t^{\frac{2}{3}} x, t^{\frac{2}{3}} x) e^{p t x} dx
\leq 
C_2\int_{-\infty}^0 \big(1 - t^{\frac{2}{3}} x \big)^\frac{1}{2} e^{p t x} dx \leq  C_3.
\end{align}
where $C_1, C_2, C_3$ only depends on $p, t_0$. Combining \eqref{eq:App}, \eqref{eq:temp23} and \eqref{eq:temp24}, we know that 
$C C_1^{-1} e^{\frac{p^3 t}{3}}  \leq \AAp(t) \leq C C_1 e^{\frac{p^3 t}{3}} + C C_3 t^{\frac{2}{3}}$. Note that $t^{\frac{2}{3}}$ can be upper bounded by a constant times $e^{\frac{p^3}{3} t}$ when $t > t_0$, we conclude Proposition \ref{prop:AAp}.
\end{proof}
\begin{proof}[Proof of Proposition \ref{prop:AApp}]
 Recall the expression of  $\AApp (\tt)$ from \eqref{eq:alphadecompose}. Since $K_{12}(x, x)$ is non-negative for all $x$, $\AApp(t)$ is lower bounded by $0$. To get the upper bound, we decompose $\AApp(t) =\frac{2^n (2n-1)!! }{ \Gamma(1-\alpha)} (\AAA_1 + \AAA_2)$ where
\begin{align*}
\numberthis \label{eq:AAA1}
\AAA_1 &= \int_0^{\infty} K_{12}(x, x) \exp\big(n t^{\frac{1}{3}} x\big) dx \int_1^{\infty} \frac{s^{-\alpha}}{\big(1 + 4s \exp(t^{\frac{1}{3}} x)\big)^{\frac{2n+1}{2}}} ds,\\
\AAA_2 &=  \int_{-\infty}^{0} K_{12}(x, x) \exp\big(n t^{\frac{1}{3}} x\big) dx \int_1^{\infty} \frac{s^{-\alpha}}{\big(1 + 4s \exp(t^{\frac{1}{3}} x)\big)^{\frac{2n+1}{2}}} ds. 
\end{align*}
Let us upper bound $\AAA_1$ and $\AAA_2$ respectively. We start with $\AAA_1$, using $1 + 4s \exp\big(t^{\frac{1}{3}} x\big) \geq 4s \exp\big(t^{\frac{1}{3}} x\big)$, 
\begin{align*}
\int_1^{\infty} \frac{s^{-\alpha}}{\big(1 + 4s \exp(t^{\frac{1}{3}} x)\big)^{\frac{2n+1}{2}}} 
\leq \int_1^{\infty} s^{-\alpha} \Big(4s \exp\big(t^{\frac{1}{3}} x\big)\Big)^{-\frac{2n+1}{2}} ds = \frac{\exp(-\frac{2n+1}{2} t^{\frac{1}{3}} x)}{2^{2n+1} \big(\frac{2n-1}{2} + \alpha\big)}.
\end{align*}
Applying this inequality to the right hand side of \eqref{eq:AAA1}, we have
$\AAone \leq C \int_0^{\infty} K_{12} (x, x) \exp(-\frac{1}{2} t^{\frac{1}{3}} x) dx.$
Using Lemma \ref{lem:k12bound} (i), for all $t > 0$, there exists a constant $C_1$ such that 
\begin{equation*}
\AAone \leq C \int_0^{\infty} \frac{e^{-\frac{2}{3} x^{\frac{3}{2}}}}{(1+x)^\frac{1}{4}} e^{-\frac{1}{2} t^{\frac{1}{3}} x} dx \leq C\int_0^{\infty} \frac{e^{-\frac{2}{3} x^{\frac{3}{2}}}}{(1+x)^{\frac{1}{4}}} dx  = C_1.
\end{equation*}
We continue to upper bound $\AAtwo$. Relaxing the integral region from $[1, \infty)$ to $[0, \infty)$, we get
\begin{align}\notag
\int_1^{\infty} \frac{s^{-\alpha}}{\big(1 + 4s \exp(t^{\frac{1}{3}} x)\big)^{\frac{2n+1}{2}}} &\leq \int_0^{\infty} \frac{s^{-\alpha}}{\big(1 + 4s \exp(t^{\frac{1}{3}} x)\big)^{\frac{2n+1}{2}}} ds
= 4^{\alpha-1}  \Be\Big(1-\alpha, \frac{2n-1 + 2\alpha}{2}\Big) e^{(\alpha-1) t^{\frac{1}{3}} x}.
\end{align}
The equality above follows from a change of variable $s \to \frac{1}{4} \exp(-t^{\frac{1}{3}} x) s$ and Lemma \ref{lem:gamma}. Due to the above display (set the product of $4^{\alpha-1}$ and the beta function to be a constant $C$)
\begin{equation*}
\AAtwo \leq   C \int_{-\infty}^0 K_{12}(x, x) \exp\big((n+\alpha-1)t^{\frac{1}{3}} x\big) dx = C \int_{-\infty}^0 K_{12}(x, x) \exp\big(pt^{\frac{1}{3}} x\big) dx.
\end{equation*}
Using Lemma \ref{lem:k12bound} to upper bound $K_{12} (x, x)$ for negative $x$, there exists a constant $C_2$ such that  for all $t > t_0$, 
\begin{equation*}
\AAtwo \leq C \int_{-\infty}^0 \sqrt{1-x} e^{p \tt^{\frac{1}{3}} x} dx = C \int_{-\infty}^0 \sqrt{1-x} e^{p t_0^{\frac{1}{3}} x} dx=  C_2.
\end{equation*}
Having $\AAone, \AAtwo$ upper bounded by a constant uniformly for $\tt> t_0$, we conclude our lemma by recalling that $\AApp (\tt)$ is a constant multiple of $\AAone + \AAtwo$.
\end{proof}
\section{Controlling the Pfaffian and Proof of Lemma \ref{lem:intdevinterchange}}
\label{sec:prepare}
 In this section, we give two upper bounds of the $L$-th Pfaffian  $\text{Pf}\big[K(x_i, x_j)\big]_{i, j=1}^L$ uniformly for all $L$. This is the main technical contribution of our paper. The purpose is two folded. First,  these  upper bounds are the crucial inputs to the proof of Proposition \ref{prop:minor} presented in the next section. Secondly, they can be used to validate the interchange of derivative, integration and summation in Lemma \ref{lem:intdevinterchange}.
\subsection{Controlling the Pfaffian}\label{sec:controlpf} We obtain two upper bounds for  $\text{Pf}\big[K(x_i, x_j)\big]_{i, j=1}^L$ for all $L$. Each upper bound has its advantage. The first upper bound has slower growth in $L$ and slower exponential decay in $x_i$ as $x_i \to \infty$. The second upper bound has faster exponential decay in $x_i$ but also a more rapid growth in $L$. For the proof of Proposition \ref{prop:minor}, we will use both of the upper bounds to control various terms in $\BBpL$ depending on how large the $L$ is. To prove these bounds, we utilize  various bounds for $K_{ij}(x, y)$, $i, j \in \{1, 2\}$ that are established in Lemma \ref{lem:kernelbound}.   
\smallskip
\\
Define  $\f_{\alpha, \beta}(x) =  e^{-\alpha x^{\frac{3}{2}}} \mathbf{1}_{\{x \geq 0\}} + (1-x)^\beta \mathbf{1}_{\{x < 0\}}$. It is clear that  $\f_{\alpha_1, \beta_1} (x) \f_{\alpha_2, \beta_2} (x) = \f_{\alpha_1 + \alpha_2, \beta_1 + \beta_2} (x)$. In addition, for $\beta_1 \leq \beta_2$, we have $\f_{\alpha, \beta_1} (x) \leq \f_{\alpha, \beta_2} (x).$
\begin{prop}\label{prop:pfbound}
There exists constant $C$ such that for all $L \in \Z_{\geq 1}$, 
\begin{enumerate}[leftmargin = 20pt, label= (\roman*)]
\item \label{item:pfbound1}
$\big|\emph{Pf}\big[K(x_i, x_j)\big]_{i, j = 1}^L \big| \leq (2L)^{L/2}  C^L \prod_{i=1}^L \f_{\frac{1}{3}, 2} (x_i)$
\item \label{item:pfbound2}
$\big|\emph{Pf}\big[K(x_i, x_j)\big]_{i, j = 1}^L \big| \leq \sqrt{(2L)!}\, C^L \prod_{i=1}^L \f_{\frac{2}{3}, 2}(x_i)$
\end{enumerate}
\end{prop}
 Looking at the growth of these upper bounds in terms of $L$, by Stirling's formula, $\sqrt{(2L)!} \sim e^{-L} (2L)^{L + \frac{1}{4}} (2\pi)^{\frac{1}{4}}$ which grows faster than $(2L)^{\frac{L}{2}}$. On the other hand, in terms of $x$, as $x_i \to \infty$, the second upper bound decays with speed $\exp(-\frac{2}{3} x_i^{\frac{3}{2}})$, which is faster than that  the first upper bound whose exponent is $\exp(-\frac{1}{3} x_i^{\frac{3}{2}})$.
\begin{proof}[Proof of Proposition \ref{prop:pfbound} \ref{item:pfbound1}]
The idea for proving Proposition \ref{prop:pfbound} \ref{item:pfbound1} is as follows. Up to a sign, the Pfaffian of a matrix equals the square root of its determinant. We apply Hadamard's inequality to upper bound the determinant in terms of the product of $\ell_\infty$-norms of each row of the matrix. Finally, we apply Lemma \ref{lem:kernelbound} to control these $\ell_\infty$-norms.  
\smallskip
\\
Now we start our proof. It is well-known that
\begin{equation}\label{eq:pfrootdet}
\Big|\pfll\Big| = \sqrt{\det\big[K(x_i, x_j)\big]_{i, j = 1}^L}.
\end{equation}
Notice that each entry $K(x_i, x_j)$ is a $2 \times 2$ matrix and $\big[K(x_i, x_j)\big]_{i, j = 1}^L$ is a $2L \times 2L$ matrix. Denote $\row_i$ to be the $i$-th row vector of this matrix. Applying Hadamard's inequality for the determinant on the right hand side above, we see that 
\begin{equation}\label{eq:hadamard}
\Big|\pfll \Big| \leq (2L)^{\frac{L}{2}} \sqrt{\prod_{i = 1}^{2L} \|\row_i\|_{\infty}}, 
\end{equation}
where $\|\cdot\|_{\infty}$ denotes the $\ell^{\infty}$-norm of a vector. It suffices to upper bound each $\|\row_i\|_{\infty}$, we do that according to the parity of $i$. For each $k = 1, \dots , L$, the vector $\row_{2k-1}$ is composed of the elements $K_{11}(x_k, x_j)$ and $K_{12}(x_k, x_j)$ for $j = 1, \dots, L$, thus 
\begin{align*}
\|\row_{2k-1}\|_{\infty} &= \max_{j = 1, \dots, L}\Big(\max\big(|K_{11}(x_k, x_j)|, |K_{12}(x_k, x_j)|\big)\Big)
\end{align*} 
Using the Lemma \ref{lem:kernelbound} \ref{item:k11} and \ref{item:k12} for  $K_{11}$ and $K_{12}$ respectively, we know that 
there exists constant $C$ such that $|K_{11} (x_k, x_j)| \leq C \f_{\frac{2}{3}, \frac{5}{4}} (x_k)$ and $|K_{12} (x_k, x_j)| \leq C \f_{\frac{2}{3}, \frac{3}{4}} (x_k)$. This implies that
$\|\row_{2k-1}\|_{\infty} \leq C \f_{\frac{2}{3}, \frac{5}{4}} (x_k).$
Similarly, the row vector $\row_{2k}$ is composed of $K_{21}(x_k, x_j)$ and $K_{22}(x_k, x_j)$, using Lemma \ref{lem:kernelbound} \ref{item:k12} and \ref{item:k22} for $K_{12}$ and $K_{22}$ respectively (note that $|K_{21}(x_k, x_j)| = |K_{12} (x_j, x_k)| \leq C\f_{0, \frac{3}{4}} (x_k)$), we get
\begin{align*}
\|\row_{2k}\|_{\infty} =  \max_{j = 1, \dots, L}\Big(\max\big(|K_{21}(x_k, x_j)|, |K_{22}(x_k, x_j)|\big)\Big) \leq C\f_{0, \frac{3}{4}} (x_k) 
\end{align*}
Inserting the upper bounds for $\|r_{2k- 1}\|$ and $\|r_{2k}\|$ into the right hand side of \eqref{eq:hadamard}, we have 
\begin{align*}
\Big|\pfll \Big| &\leq (2L)^{\frac{L}{2}} C^L \sqrt{\prod_{k=1}^L \f_{\frac{2}{3}, \frac{5}{4}}(x_k)} \cdot \sqrt{\prod_{k=1}^L \f_{0, \frac{3}{4}} (x_k)} \leq (2L)^{\frac{L}{2}} C^L \prod_{k=1}^L \f_{\frac{1}{3}, 2} (x_k).
\end{align*}
The last equality follows from $\sqrt{\f_{\frac{2}{3}, \frac{5}{4}} (x_k) \f_{0, \frac{3}{4}} (x_k)} = \f_{\frac{1}{3}, 1} (x_k) \leq \f_{\frac{1}{3}, 2}(x_k)$. This completes our proof.
\end{proof}
 To prove Proposition \ref{prop:pfbound} \ref{item:pfbound2}, we upper bound the determinant in a different way. Instead of using Hadamard's inequality, we work with the permutation expansion of the determinant and seek to upper bound each term therein. We introduce some notations. Rewrite the $2L \times 2L$ matrix $\big[K(x_i, x_j)\big]_{i, j = 1}^{L}$  as $\big[\D(i, j)\big]_{i, j = 1}^{2L}$ in a way that for all $i, j \in \{1, \dots, L\}$,
\begin{align*}
&\D(2i-1, 2j-1) = K_{11}(x_i, x_j),\qquad \D(2i-1, 2j) = K_{12}(x_i, x_j),\\
&\D(2i, 2j-1) = K_{21}(x_i, x_j),\qquad\qquad \D(2i, 2j) = K_{22}(x_i, x_j). 
\end{align*}
Note that we suppress the dependence on $x$ in the notation of $\D$.
Define the maps $\theta: \{1, \dots, 2L \} \to \{1, \dots, L\}$ and $\tau: \{1, \dots, 2L \} \to \{0, 1\}$ such that  $\theta (n) = \lfloor \frac{n}{2}\rfloor$ and  $\tau(n) = n - 2 \lfloor n/2 \rfloor$. It is clear that $$(\theta, \tau) : \{1, \dots, 2L\} \to \{1, \dots, L\} \times \{0, 1\}$$ 
is a bijection. 
\begin{lemma}\label{lem:Dentrybound}
There exists a constant $\C$ such that for arbitrary $L$ and $i, j \in \{1, \dots, 2L\}$,  $$|\D(i, j)| \leq \C\f_{\frac{2}{3}\tau(i), \frac{3}{4}} (x_{\theta(i)}) \f_{\frac{2}{3}\tau(j), \frac{3}{4}} (x_{\theta(j)}) .$$
\end{lemma}
\begin{proof}
We divide our proof of the above inequality  into four cases. \emph{Case $1$}: $i, j$ are both odd. \emph{Case $2$}: $i, j$ are both even. \emph{Case $3$}: $i$ is odd and $j$ is even. \emph{Case $4$}: $i$ is even and $j$ is odd.
\smallskip
\\
\textbf{Case 1.} $i, j$ are both odd. Then $\tau(i) = \tau(j) = 1$ and $\D(i, j) = K_{11} (x_{\theta(i)}, x_{\theta(j)})$. By Lemma \ref{lem:kernelbound} \ref{item:k11}, 
$$|\D(i, j)| \leq C\f_{\frac{2}{3}, \frac{3}{4}} (x_{\theta(i)}) \f_{\frac{2}{3}, \frac{3}{4}} (x_{\theta(j)}) $$
\textbf{Case 2.} $i, j$ are both even. Then $\tau(i) = \tau(j) = 0$ and $\D(i, j) = K_{22}(x_{\theta(i)}, x_{\theta(j)})$. By Lemma \ref{lem:kernelbound} \ref{item:k22}, 
$$|\D(i, j)| \leq C \f_{0,\frac{3}{4}}(x_{\theta(i)}) \leq  C \f_{0, \frac{3}{4}}(x_{\theta(i)}) \f_{0, \frac{3}{4}}(x_{\theta(j)}).$$
where the second inequality above follows from $F_{0, \frac{3}{4}}(x) = 1 + (1-x)^{\frac{3}{4}} \mathbf{1}_{\{x \leq 0\}} \geq 1$. 
\smallskip
\\
\textbf{Case 3.} $i$ is odd and $j$ is even. Then $\tau(i) = 1$, $\tau(j) =0$ and $\D(i, j) = K_{12}(x_{\theta(i)}, x_{\theta(j)})$. Using Lemma \ref{lem:kernelbound} \ref{item:k12}, $$|\D(i, j)| \leq C \f_{\frac{2}{3}, \frac{3}{4}} (x_{\theta(i)}) \leq C \f_{\frac{2}{3}, \frac{3}{4}} (x_{\theta(i)}) \f_{0, \frac{3}{4}} (x_{\theta(j)}).$$
\textbf{Case 4}. $i$ is even and $j$ is odd. Then $\tau(i) = 0$, $\tau(j) = 1$ and $\D(i, j) = K_{21}(x_{\theta(i)}, x_{\theta(j)})$. This follows from \textbf{Case 3} and the fact $K_{21}(x, y) = -K_{12} (y, x)$. 
\smallskip
\\
Since our discussion covers all the cases, we conclude our proof of the lemma. 
\end{proof}
\begin{proof}[Proof of Proposition \ref{prop:pfbound} \ref{item:pfbound2}]
Referring to \eqref{eq:pfrootdet} and permutation expansion of the determinant, 
\begin{equation}\label{eq:pfdet}
\Big(\pfll\Big)^2 = \det\big[\D(x_i, x_j)\big]_{i, j = 1}^{2L} = \sum_{\sigma \in S_{2L}} \prod_{i=1}^L \D(i, \sigma(i)) 
\end{equation}
where $S_{2L}$ is the permutation group of $\{1, \dots, 2L\}$. Applying Lemma \ref{lem:Dentrybound} for any permutation $\sigma \in S_{2L}$, 
\begin{align*}
\Big|\prod_{i=1}^{2L} \D\big(i, \sigma(i)\big)\Big| \leq C^{2L} \prod_{i=1}^{2L} \f_{\frac{2}{3}\tau(i), \frac{3}{4}} (x_{\theta(i)}) \f_{\frac{2}{3}\tau(\sigma(i)), \frac{3}{4}} (x_{\theta(\sigma(i))})
&= 
C^{2L} \prod_{i=1}^{2L} \f_{\frac{2}{3}\tau(i), \frac{3}{4}} (x_{\theta(i)})  \prod_{i=1}^{2L} \f_{\frac{2}{3}\tau(i), \frac{3}{4}} (x_{\theta(i)})
\end{align*}
Using the bijectivity of $(\tau, \theta): \{1, 2\dots, 2L\} \to \{1, \dots, L\} \times \{0, 1\}$, we see that for arbitrary $\sigma \in S_{2L}$,
\begin{equation*}
\prod_{i=1}^{2L} \f_{\frac{2}{3}\tau(\sigma(i)), \frac{3}{4}} (x_{\theta(\sigma(i))}) = \prod_{i=1}^{2L} \f_{\frac{2}{3}\tau(i), \frac{3}{4}} (x_{\theta(i)}) = \prod_{i=1}^{L} \f_{\frac{2}{3}, \frac{3}{4}} (x_{i}) \f_{0, \frac{3}{4}} (x_i) = \prod_{i=1}^L \f_{\frac{2}{3}, \frac{3}{2}}(x_i)
\end{equation*} 
This implies that for all permutation $\sigma \in S_{2L}$,  the absolute value of $\prod_{i=1}^{2L} \D\big(i, \sigma(i)\big)$ can be upper bounded by $C^{2L} \big(\prod_{i=1}^L \f_{\frac{2}{3}, \frac{3}{2}} (x_i)\big)^2$. Referring to \eqref{eq:pfdet}, since there are $(2L)!$ terms in the summation on the right hand side,
\begin{align*}
\Big(\pfll \Big)^2 &\leq (2L)! \max_{\sigma \in S_{2L}}\Big( \prod_{i=1}^{2L} \D(i, \sigma(i))\Big) \leq  (2L)! C^{2L} \Big(\prod_{i=1}^L \f_{\frac{2}{3}, \frac{3}{2}} (x_i)\Big)^2
\end{align*}
Taking the square root for both sides above and using $\f_{\frac{2}{3}, \frac{3}{2}} (x_i) \leq \f_{\frac{2}{3}, 2} (x_i)$, this completes the proof.
\end{proof}
\subsection{Proof of Lemma \ref{lem:intdevinterchange}}\label{sec:dis}
In this subsection, we devote to justify that we can interchange the order of derivative, integration and summation for the right hand side \eqref{eq:dis} and provide a proof of Lemma \ref{lem:intdevinterchange}. For the ensuing discussion, we denote $\partial_s^k \phi_{s, t} (x)$ by $\phi_{s, t}^{(k)} (x)$. In particular, when $k = 0$, $\phi_{s, t}^{(k)} (x)$ coincides with $\phi_{s, t}(x)$. 
\begin{lemma}\label{lem:phibound}
Fix $k$. Recall that $\phi_{s, t} (x) = \frac{1}{\sqrt{1 + 4s \exp(t^{1/3} x)}} - 1$, there exists $C = C(k)$ such that  for all $s \geq 0$, 
\begin{equation*}
|\phi^{(k)}_{s, t}(x)|  \leq 1 \wedge 2 s e^{t^{\frac{1}{3}} x} \quad \emph{ if } k = 0; \qquad |\phi^{(k)}_{s, t}(x)| \leq C \big(e^{k t^{\frac{1}{3}} x} \wedge s^{-k}\big) \quad \emph{ if } k \in \Zp,
\end{equation*}
\end{lemma}
\begin{proof}
It is easily verified that for all $y \geq 0$, $1 - \frac{1}{\sqrt{1 + y}}  \leq 1 \wedge \frac{1}{2} y$. Taking $y = 4s \exp(\tt^{\frac{1}{3}} x) $ implies the first inequality. For the second inequality, we compute $\phi_{s, t}^{(k)} (x) = (-2)^k (2k-1)!! \frac{\exp(kt^{\frac{1}{3} }x)}{(1 + 4s\exp(t^{\frac{1}{3}}x))^{\frac{2k+1}{2}}}$. Lower bounding $(1 + 4s \exp(t^{\frac{1}{3}} x))^{\frac{2k+1}{2}}$ by $1$, we get $|\phi_{s,t}^{(k)} (x)| \leq 2^k (2k-1)!! \exp(k t^{\frac{1}{3}} x)$. On the other hand, we have 
\begin{equation*}
|\phi_{s, t}^{(k)} (x) | \leq 2^k (2k-1)!! \frac{\exp(k t^{\frac{1}{3} x})}{\big(1 + 4s \exp(t^{\frac{1}{3}} x)\big)^k} \leq s^{-k} 2^k (2k-1)!!
\end{equation*}
This completes our proof.
\end{proof}
 Let us introduce a few notations. Define 
$
\M(L, n) = \big\{\vec{m} = (m_1, \dots, m_L) \in \Zn^L, \sum_{i=1}^L m_i = n\big\}$ and for $\vec{m} \in \M(L, n)$,  we set  $\binom{n}{\vec{m}} = \frac{n!}{\prod_{i=1}^L m_i!}$.
\begin{lemma}\label{lem:pftimesdev}
Fix $n \in \Zn$ and $t > 0$, there exists a constant $C = C(n, t)$ such that for all $L \in \Zp$ and $s \in [0, 1]$,  
\begin{equation*}
\bigg|\int_{\RL} \emph{Pf}\big[K(x_i, x_j)\big]_{i, j = 1}^L \partial_s^n \Big(\prod_{i=1}^L \phi_{s, t}(x_i)\Big) dx_1 \dots dx_L\bigg| \leq (2L)^{\frac{L}{2}} C^L.
\end{equation*}
\end{lemma}
\begin{proof}
Throughout the proof, we write $C = C(n, t)$. By Leibniz's rule, 
\begin{equation}\label{eq:leibniz}
\partial_s^n \Big(\prod_{i=1}^L \phi_{s,t}(x_i)\Big) = \sum_{\vec{m} \in \M(L, n)} \binom{n}{\vec{m}} \prod_{i=1}^L \phi^{(m_i)}_{s, t} (x_i).
\end{equation}
According to Lemma \ref{lem:phibound}, there exists constant $C_1$ such that for each $0 \leq m_i \leq n$ and $s \in [0, 1]$
\begin{equation*}
|\phi_{s, t}^{(m_i)} (x_i)| \leq 2s e^{t^{\frac{1}{3}} x} \leq 2 e^{t^{\frac{1}{3}} x} \quad \text{ if } m_i = 0; \qquad |\phi_{s, t}^{(m_i)} (x_i)| \leq C_1 e^{m_i t^{\frac{1}{3}} x}\quad  \text{ if } 1 \leq m_i \leq n.
\end{equation*}
Hence, we have $|\phi_{s, t}^{(m_i)} (x_i)| \leq C_1 \big(\exp(t^{\frac{1}{3}} x) \vee \exp(t^{\frac{1}{3}} n x)\big)$. Taking the absolute value for both sides of \eqref{eq:leibniz}  and applying triangle inequality,
\begin{equation}
\label{eq:temp1}
\Big|\partial_s^n \Big(\prod_{i=1}^L \phi_{s,T}(x_i)\Big)\Big| \leq C_1^L \sum_{\vec{m} \in\M(L, n)}  \binom{n}{\vec{m}}  \prod_{i=1}^L \Big( e^{t^{\frac{1}{3}} x_i} \vee e^{n t^{\frac{1}{3}} x_i}\Big) = C_1^L \prod_{i=1}^L \Big( e^{t^{\frac{1}{3}} x_i} \vee e^{n t^{\frac{1}{3}} x_i}\Big) \bigg(\sum_{\vec{m} \in\M(L, n)}  \binom{n}{\vec{m}}\bigg)
\end{equation}
It suffices to show that there exists a constant $C_2$ such that for all $L \in \Zp$ \begin{equation}\label{eq:temp2}
\sum_{\vec{m} \in \M(L,n)} \binom{n}{\vec{m}} \leq C_2^L.
\end{equation}
Once this shown, by \eqref{eq:temp1} and triangle inequality, we see that $\big|\partial_s^n \big(\prod_{i=1}^L \phi_{s,T}(x_i)\big)\big|$ is upper bounded by $(C_1 C_2)^L \prod_{i=1}^L \big( e^{t^{\frac{1}{3}} x_i} \vee e^{n t^{\frac{1}{3}} x_i}\big)$. Applying Proposition \ref{prop:pfbound} yields 
\begin{equation*}
\bigg|\int_{\RL} \text{Pf}\big[K(x_i, x_j)\big]_{i, j = 1}^L \partial_s^n \Big(\prod_{i=1}^L \phi_{s, t}(x_i)\Big) dx\bigg| \leq (2L)^{\frac{L}{2}} C^L \Big(\int_\R \f_{\frac{1}{3}, 2} (x) \big(e^{t^{\frac{1}{3}} x} \vee e^{n t^{\frac{1}{3}} x}\big) dx\Big)^L.
\end{equation*}
Since $t$ is fixed, the integrand on the right hand side above  decays super-exponentially at $+\infty$ and exponentially at $-\infty$, hence is integrable. The value of the integration is a constant that only depends on $n, t$. This completes our proof of the lemma.
\smallskip
\\
It remains to prove \eqref{eq:temp2}.
Let $\# A$ be the number of elements in $A$. Note that  
\begin{equation}\label{eq:MLn}
\#\M(L, n) = \#\{\vec{m} = (m_1, \dots, m_L) \in \Zn^L, \sum_{i=1}^L m_i = n\}  \leq L \#\M(L, n-1).
\end{equation}
Iterating this inequality yields $\#\M(L, n) \leq L^n$. In addition, for each $\vec{m} \in \M(L,n)$, $\binom{n}{\vec{m}}$ is upper bounded by $n!$. We can find a large constant 
$C_2 = C_2 (n)$ such that for all $L \geq 1$,  
\begin{equation*}
\sum_{\vec{m} \in\M(L, n)}  \binom{n}{\vec{m}} \leq n! \# \M(L, n) \leq n! L^n \leq C_2^L. \qedhere
\end{equation*}
\end{proof}
 The next two propositions validate that we can interchange the order of derivative, summation and integral.
\begin{prop}\label{lem:derivativeintegral}
For every fixed $n, L \in \Z_{\geq 1}$, $s \in [0, 1]$ and $t > 0$, we have 
\begin{align*}
\partial_s^n \int_{\RL} \emph{Pf}\big[K(x_i, x_j)\big]_{i, j=1}^L  \prod_{i = 1}^{L} \phi_{s, t}(x_i) dx_i
= \int_{\RL} \emph{Pf}\big[K(x_i, x_j)\big]_{i, j = 1}^L \partial_s^n \Big(\prod_{i=1}^L \phi_{s, t}(x_i)\Big) dx_1 \dots dx_L
\end{align*}
\end{prop}
\begin{proof}
The proof follows from the dominated convergence theorem. It suffices to show that we can find an integrable function $G(x_1, $\dots$, x_L)$ such that for all $s \in [0, 1]$,  
\begin{equation}\label{eq:Gdominate}
\Big|\text{Pf}\big[K(x_i, x_j)\big]_{i, j =1}^L \partial_s^n \Big(\prod_{i=1}^L \phi_{s, t}(x_i)\Big)\Big| \leq G(x_1, \dots, x_L)
\end{equation}
By the argument in the proof of Lemma \ref{lem:pftimesdev}, we can take $$G(x_1, \dots, x_L) = (2L)^{\frac{L}{2}} C^L \prod_{i=1}^L \f_{\frac{1}{3}, 2} (x_i) \big(e^{t^{\frac{1}{3}} x_i} \vee e^{n t^{\frac{1}{3}} x_i}\big)$$
where $C$ is the constant in Lemma \ref{lem:pftimesdev}. Since $G$ is integrable and satisfy \eqref{eq:Gdominate}, this completes our proof. 
\end{proof}
\begin{prop}\label{lem:summationintegral}
For fixed $n \in \Zn, t > 0$ and $s \in [0, 1]$ we have the following interchange of differentiation and summation holds
\begin{align*}
&\partial_s^n \Big(\sum_{L=1}^{\infty} \frac{1}{L!} \int_{\RL} \emph{Pf} \big[K(x_i, x_j)\big]_{i, j=1}^L \prod_{i=1}^L \phi_{s, t}(x_i) dx_i\Big) = \sum_{L=1}^{\infty} \partial_s^n \Big(\int_{\RL} \emph{Pf}\big[K(x_i, x_j)\big]_{i, j =1}^L  \prod_{i=1}^L \phi_{s, t}(x_i) dx_i\Big).
\end{align*}
\end{prop}
\begin{proof}
A sufficient condition for the interchange of the order of derivative and infinite summation is that (see \cite[Proposition 4.2]{DT19}), 
\begin{enumerate}[leftmargin=20pt, label=(\roman*)]
\item \label{item:sumconverge} $\sum_{L=1}^{\infty} \frac{1}{L!} \int_{\RL} \text{Pf} \big[K(x_i, x_j)\big]_{i, j=1}^L \prod_{i=1}^L \phi_{s, t}(x_i) dx_i$ converges pointwisely for $s \in [0, 1]$.
\item \label{item:dsumconverge} For all $n \in \Zp$, $\sum_{L=1}^{\infty} \frac{1}{L!} \partial_s^n \big(\int_{\RL} \text{Pf} \big[K(x_i, x_j)\big]_{i, j=1}^L \prod_{i=1}^L \phi_{s, t}(x_i) dx_i\big)$ converges uniformly for $s \in [0, 1]$.
\end{enumerate}
Applying Proposition \ref{lem:derivativeintegral}, for (ii), we can place $\partial_s^n$ inside the integral. Both \ref{item:sumconverge} and \ref{item:dsumconverge} then follow from Lemma \ref{lem:pftimesdev} and the convergence of  $\sum_{L=1}^\infty \frac{C^L (2L)^{\frac{L}{2}}}{L!}$.
\end{proof}
\begin{proof}[Proof of Lemma \ref{lem:intdevinterchange}]
It is enough to show \eqref{eq:laplace1}. Combining Proposition \ref{lem:derivativeintegral} and \ref{lem:summationintegral}, we know that 
\begin{align*}
&\int_0^{1} s^{-\alpha} \partial_s^n \Big(\sum_{L = 1}^\infty \int_{\RL} \pfll  \prod_{i = 1}^L  \phi_{s, t}(x_i) dx_i\Big) ds\\ 
&= \int_0^{1} s^{-\alpha}  \bigg(\sum_{L = 1}^\infty \int_{\RL} \pfll \partial_s^n \Big(\prod_{i = 1}^L  \phi_{s, t}(x_i)\Big) dx_1 \dots dx_L\bigg) ds
\end{align*}
Thus, what remains to prove is that 
\begin{align*}
&\int_0^{1} s^{-\alpha}  \bigg(\sum_{L = 1}^\infty \int_{\RL} \pfll \partial_s^n \Big(\prod_{i = 1}^L  \phi_{s, t}(x_i)\Big) dx_1 \dots dx_L\bigg) ds \\
&=\sum_{L = 1}^\infty\int_0^{1} s^{-\alpha}  \bigg( \int_{\RL} \pfll \partial_s^n \Big(\prod_{i = 1}^L  \phi_{s, t}(x_i)\Big) dx_1 \dots dx_L\bigg) ds
\end{align*}
This is justified via the dominated convergence theorem, using Lemma \ref{lem:pftimesdev} and the convergence of  $\sum_{L=1}^\infty \frac{C^L (2L)^{\frac{L}{2}}}{L!}$.
\end{proof}

\section{Proof of Proposition \ref{prop:minor}}
\label{sec:Bp}
 In this section, we prove Proposition \ref{prop:minor}.
The main inputs will be  Proposition \ref{prop:pfbound} and \ref{prop:integrationbound}.
\smallskip
\\
Define $\VV_n(x) = nx - \frac{1}{3} x^{\frac{3}{2}}$ and
$\UU_n (x) = n x - \frac{2}{3} x^{3/2}$. By straightforward calculus, $\VV_n (\sigma \wedge 4n^2)$ (resp. $\UU_n (\sigma \wedge n^2)$) is the maximum of $\VV_n$ (resp. $\UU_n$) over $x \in [0, \sigma]$. Recall from the beginning of Section \ref{sec:prepare} that $\f_{\alpha, \beta}(x) =  e^{-\alpha x^{\frac{3}{2}}} \mathbf{1}_{\{x \geq 0\}} + (1-x)^\beta \mathbf{1}_{\{x < 0\}}$. We start with the following lemma.
\begin{lemma}\label{lem:VnUnlaplace}
Fix $n \in \Zp$ and $t_0 > 0$, there exists $C = C(n, t_0)$ such that for all $\sigma \geq 0$ and $t > t_0$,
\begin{enumerate}[leftmargin = 20pt, label = (\roman*)]
\item \label{item:Vnlaplace} $\int_{-\infty}^{\sigma} \exp(tnx) \f_{\frac{1}{3}, 2} (t^{\frac{2}{3}} x)dx \leq C t^{-\frac{1}{2}} \exp(t V_n(\sigma \wedge 4n^2))$
\item \label{item:Unlaplace} $\int_{-\infty}^{\sigma} \exp(tnx)  \f_{\frac{2}{3}, 2} (t^{\frac{2}{3}} x) \leq C t^{-\frac{1}{2}} \exp(t U_n (\sigma \wedge n^2))$
\end{enumerate}
\end{lemma}
\begin{proof}
Let us first demonstrate \ref{item:Vnlaplace}. 
Decompose 
\begin{equation}\label{eq:laplacedecompose}
\int_{-\infty}^\sigma \exp(tnx) \f_{\frac{1}{3}, 2} (t^{\frac{2}{3}} x) dx = \int_{-\infty}^0 \exp(tnx) \f_{\frac{1}{3}, 2} (t^{\frac{2}{3}} x) dx + \int_{0}^\sigma \exp(tnx) \f_{\frac{1}{3}, 2} (t^{\frac{2}{3}} x) dx
\end{equation}
Since $F_{\alpha, \beta}(x)$ equals $(1-x)^{\beta}$ when $x$ is negative, we can rewrite the first term on the right hand side in the above display into $\int_{-\infty}^0 \exp(tn x) (1-t^{\frac{2}{3}} x)^2 dx$. We have 
$$\int_{-\infty}^0 e^{tnx} (1-t^{\frac{2}{3}} x)^2 dx = t^{-1} \int_{0}^{\infty} e^{-nx} (1+t^{-\frac{1}{3}} x)^2 dx \leq t^{-1} \int_0^{\infty} e^{-nx} (1 +t_0^{-1/3} x)^2 dx \leq Ct^{-1},$$ 
where the first equality is due to a change of variable $x \to -t^{-1} x$, the second equality follows our condition $t \geq t_0$ and the third is due to $\int_0^{\infty} e^{-nx} (1 + t_0^{-\frac{1}{3}x})^2 dx$ is finite constant only depending on $n, t_0$.  
We have shown that the first term on the right hand side of \eqref{eq:laplacedecompose} is upper bounded by $C t^{-1}$ for some $C = C(n, t_0)$. 
Since $V_n (\sigma \wedge 4n^2)$ is non-negative, which implies that $t^{-1} \leq C t^{-\frac{1}{2}}  \exp(t \VV_n (\sigma \wedge 4n^2))$ for $t \geq t_0$. To prove \ref{item:Vnlaplace},  it suffices to prove the second term on the right hand side of \eqref{eq:laplacedecompose} is also upper bounded by $C t^{-\frac{1}{2}}  \exp(t \VV_n (\sigma \wedge 4n^2))$. Note that when $x \geq 0$, $F_{\frac{1}{3}, 2} (t^{\frac{2}{3}} x) = \exp(-\frac{1}{3}t x^{\frac{3}{2}} )$ which yields $\int_0^{\sigma} \exp(tnx) \f_{\frac{1}{3}, 2} (t^{\frac{2}{3}} x) dx = \int_0^{\sigma} \exp(t V_n(x)) dx$. Thus we only need to show that there exists a constant $C = C(n, t_0)$ such that for $t > t_0$
\begin{equation}\label{eq:needtoshow}
\int_0^{\sigma} e^{t \VV_n (x)} dx \leq C t^{-\frac{1}{2} } e^{t V_n(\sigma \wedge 4n^2)}.
\end{equation}
To this aim, we split our discussion into the following three cases.
\smallskip
\\
\textbf{Case 1.} $\sigma \in [0, n^2]$. Since $V_n(x)$ is concave on $\sigma \in [0, n^2]$. Hence, for all $x \in [0, \sigma]$,
$$\VV_n (x) \leq \VV_n(\sigma) + \VV_n'(\sigma)(x- \sigma) = \VV_n (\sigma) + (n - \frac{1}{2} \sqrt{\sigma}) (x-\sigma) \leq \VV_n(\sigma) + \frac{n}{2}(x-\sigma).$$
The last inequality above follows since $x \leq \sigma$ and $n - \frac{1}{2}\sqrt{\sigma} \geq n/2$. Using the displayed inequality above,
\begin{equation*}
\int_0^{\sigma} e^{t \VV_n (x)} dx \leq \int_0^\sigma e^{t (\VV_n (\sigma) + \frac{n}{2} (x-\sigma))} dx = e^{t \VV_n (\sigma)} \int_0^\sigma e^{\frac{nt}{2}(x-\sigma)} dx \leq \frac{2}{nt}e^{t \VV_n (\sigma)}.
\end{equation*}
which implies \eqref{eq:needtoshow}.
\smallskip
\\
\textbf{Case 2.} $\sigma \in [n^2, 4n^2]$. First via a change of variable $x = r^2$, 
$\int_0^{\sigma} e^{t V_n (x)} dx = \int_0^{\sqrt{\sigma}}  2r e^{t V_n (r^2)} dr$.  Therefore, we only need to prove \eqref{eq:needtoshow} with $\int_0^{\sqrt{\sigma}}  2r e^{t V_n (r^2)} dr$ in place of $\int_0^{\sigma} e^{t V_n (x)} dx$.
Since $r \leq \sqrt{\sigma} \leq 2n$,  
\begin{equation*}
V_n (r^2) - \VV_n (\sigma) = n(r^2 - \sigma) - \frac{1}{3} (r^{3} - \sigma^{\frac{3}{2}}) \leq \frac{\sqrt{\sigma}}{2} (r^2-\sigma) - \frac{1}{3} (r^3 - \sigma^{\frac{3}{2}}) = - \frac{1}{3} (r - \sqrt{\sigma})^2 (r + \frac{1}{2} \sqrt{\sigma}).
\end{equation*}
Since $\sigma \geq n^2$ and $r \geq 0$, from the above displayed inequality, $\VV_n (r^2) \leq \VV_n(\sigma) - \frac{n}{6} (r- \sqrt{\sigma})^2  $.
Consequently, 
\begin{equation*}
\int_0^{\sqrt{\sigma}} 2r e^{tV_n (r^2)} dr \leq \int_0^{\sqrt{\sigma}} 2r e^{t \VV_n (\sigma)} e^{-\frac{tn}{6} (r-\sqrt{\sigma})^2} dr  = e^{t V_n (\sigma)} \int_0^{\sqrt{\sigma}} 2r e^{-\frac{t}{6} (r-\sqrt{\sigma})^2} dr \leq \frac{C}{\sqrt{t}} e^{t \VV_n(\sigma)}.
\end{equation*}
This completes the proof of \eqref{eq:needtoshow}. In the last inequality above, we used $\sqrt{\sigma} \leq 2n$ and $t \geq t_0$, which implies that  $$\int_0^{\sqrt{\sigma}} 2r e^{-\frac{t}{6} (r-\sqrt{\sigma})^2} dr \leq \int_{-\infty}^{\infty} 2(r+\sqrt{\sigma}) e^{-\frac{t}{6} r^2} dr \leq \int_{-\infty}^{\infty} 2(r +2n) e^{-\frac{t}{6}r^2} dr = C_1 t^{-1} + C_2 t^{-1/2} \leq Ct^{-1/2}.$$
\smallskip
\\
\textbf{Case 3.} $\sigma > 4n^2$. As illustrated in \textbf{Case 2}, we only need to show that for $t > t_0$, $$\int_0^{\sqrt{\sigma}} 2r e^{t\VV_n (r^2)} dr \leq C t^{-\frac{1}{2} } e^{t V_n(\sigma \wedge 4n^2)}$$ 
Note that 
$\VV_n (r^2) - \frac{4}{3}n^3 = nr^2 - \frac{1}{3} r^3 - \frac{4}{3}n^3= - \frac{1}{3} (r - 2n)^2 (n+r) \leq -\frac{n}{3} (r-2n)^2$. This implies 
\begin{equation*}
\int_0^{\sqrt{\sigma}} 2 r e^{t\VV_n (r^2)} dr \leq \int_0^{\sqrt{\sigma}} 2r e^{\frac{4}{3}t n^3}  e^{-\frac{n}{3} t(r-2n)^2} dr \leq  e^{\frac{4}{3} n^3 t} \int_0^{\sqrt{\sigma}} 2r e^{-\frac{n}{3} t (r-2n)^2} dr \leq \frac{C}{\sqrt{t}} e^{\frac{4}{3} n^3 t}.
\end{equation*}
The last inequality is due to a similar argument as in \textbf{Case 2} above. Since $\sigma > 4n^2$, we have $\VV_n (\sigma \wedge 4n^2) = \VV_n (4n^2) = \frac{4}{3} n^3$, thus we showed \eqref{eq:needtoshow}. So far, we complete the proof of \ref{item:Vnlaplace}.
\smallskip
\\ 
The proof of \ref{item:Unlaplace} will be rather similar, instead of \eqref{eq:needtoshow}, the key step is to show that there exists $C = C(n, t_0)$ such that for $t \geq t_0$, $\int_0^\sigma \exp\big(t U_n(x)\big) dx \leq C t^{-\frac{1}{2}} \exp\big(t U_n(\sigma \wedge n^2)\big)$.  We skip the details. 
\end{proof}
\begin{prop}\label{prop:integrationbound}
For fixed $n \in \Zp$ and $t_ 0 > 0$, there exists constant $C = C(n, t_0)$ such that for every  $\sigma \geq 0$ and $t > t_0$, 
\begin{enumerate}[leftmargin = 20pt, label= (\alph*)]
\item \label{item:g11}
$\int_{-\infty}^{\infty}  |\phi_{e^{-t \sigma}, t}(x)| \ft(x)dx  \leq C t^{\frac{1}{6}}\exp\big(t \VV_1 (\sigma \wedge 4) - t \sigma\big)$
\item \label{item:g1n}
$\int_{-\infty}^{\infty} |\phi^{(n)}_{e^{-t \sigma}, t} (x)| \ft(x) dx \leq  C t^{\frac{1}{6}}\exp\big(t \VV_n (\sigma \wedge 4 n^2)\big)$
\item \label{item:g21}
$\int_{-\infty}^{\infty}  |\phi_{e^{-t \sigma}, t}(x)| \ftt(x) dx  \leq C t^{\frac{1}{6}}\exp\big(t \UU_1 (\sigma \wedge 1) - t \sigma\big)$
\item \label{item:g2n}
$\int_{-\infty}^{\infty} |\phi^{(n)}_{e^{-t \sigma}, t} (x)| \ftt(x) dx \leq  C t^{\frac{1}{6}}\exp\big(t \UU_n (\sigma \wedge n^2)\big).$
\end{enumerate}
\end{prop}
\begin{proof}
We first prove \ref{item:g11}. Via a change of variable $x \to t^{\frac{2}{3}} x$, 
$$\int_{-\infty}^{\infty} |\phi_{e^{-t\sigma}, t}(x)| \ft(x) dx= t^{\frac{2}{3}}\int_{-\infty}^\infty |\phi_{e^{-t\sigma}, t}(t^{\frac{2}{3}}x)|   \ft(t^{\frac{2}{3}}x) dx.$$
We decompose the integral region on the right hand side above into $(-\infty,\sigma) \cup (\sigma, \infty)$ and write 
\begin{equation*}
\int_{-\infty}^\infty |\phi_{e^{-t\sigma}, t}(t^{\frac{2}{3}}x)|   \ft(t^{\frac{2}{3}}x) dx =  \Big(\int_{-\infty}^\sigma  + \int_{\sigma}^\infty \Big) |\phi_{e^{-t\sigma}, t}(t^{\frac{2}{3}}x)|   \ft(t^{\frac{2}{3}}x) dx = \As + \Ass
\end{equation*}
We provide upper bounds for $\As$ and $\Ass$ respectively. By Lemma \ref{lem:phibound}, $|\phi_{e^{-t\sigma}, t}(t^{\frac{2}{3}}x)| \leq 2 \exp\big(t(x - \sigma)\big)$, so
\begin{align}\label{eq:E1} 
\As \leq 2 e^{-t\sigma} \int_{-\infty}^\sigma e^{tx} \f_{\frac{1}{3}, 2} (x) dx \leq C t^{-\frac{1}{2}} e^{t V_1(\sigma \wedge 4) - t \sigma}.
\end{align}
where the last inequality above is due to Lemma \ref{lem:VnUnlaplace} \ref{item:Vnlaplace} (setting $n = 1$ therein). On the other hand, since $\f_{\frac{1}{3}, 2} (t^{\frac{2}{3}}x) = \exp(-\frac{1}{3} t x^{\frac{3}{2}})$ when $x \geq 0$,
\begin{equation}\label{eq:E2}
\Ass = \int_{\sigma}^\infty |\phi_{e^{-t\sigma}, t}(t^{\frac{2}{3}}x)|   e^{-\frac{1}{3}t x^{3/2}} dx \leq \int_{\sigma}^\infty e^{-\frac{1}{3}t x^{3/2}} dx = t^{-\frac{2}{3}}\int_{t^{\frac{2}{3}}\sigma}^{\infty}e^{-\frac{1}{3} x^{3/2}} dx  \leq C t^{-\frac{2}{3}} e^{-\frac{1}{3} t \sigma^{\frac{3}{2}}}
\end{equation}
The first inequality above is due to Lemma \ref{lem:phibound}, which yields $|\phi_{e^{-t\sigma}, t} (t^{\frac{2}{3}} x)| \leq 1$. The second inequality follows from a change of variable $x \to t^{-\frac{2}{3}} \sigma$, and the last inequality is due to the fact
$\int_y^{\infty} \exp(-\frac{1}{3} x^{\frac{3}{2}}) dx\leq C \exp(-\frac{1}{3} y^{\frac{3}{2}})$, which holds for all $y \geq 0$. Combining \eqref{eq:E1} and \eqref{eq:E2} and recall that $\int_{-\infty}^{\infty} \phi_{e^{-t\sigma}, t}(x) \ft(x) dx = t^{\frac{2}{3}} (\As + \Ass)$, we obtain
\begin{equation*}
\int_{-\infty}^\infty |\phi_{e^{-t\sigma}, t} (x)| \f_{\frac{1}{3}, 2} (x) dx \leq Ct^{\frac{2}{3}} \big(t^{-\frac{1}{2}}\exp(tV_1 (\sigma \wedge 4) - t\sigma) + t^{-\frac{2}{3}} e^{-\frac{1}{3} t \sigma^{\frac{3}{2}}}\big).
\end{equation*}
Since $V_1(\sigma \wedge 4)$ is the maximum of $V_1 (x)$ for $x \in [0, \sigma]$, $V_1 (\sigma \wedge 4) - \sigma \geq V_1 (\sigma) - \sigma = -\frac{1}{3} \sigma^{\frac{3}{2}}$, so the first term on the right hand side above dominates, this completes the proof of \ref{item:g11}.
\bigskip
\\
For the proof of \ref{item:g1n}, via a change of variable $x \to t^{\frac{2}{3}} x$, 
\begin{equation*}
\int_{-\infty}^{\infty} |\phi^{(n)}_{e^{-t \sigma}, t} (x)| \ft(x) dx  = t^{\frac{2}{3}} \int_{-\infty}^\infty |\phi^{(n)}_{e^{-t\sigma}, t} (t^{\frac{2}{3}} x)| \f_{\frac{1}{3}, 2} (t^{\frac{2}{3}} x) dx
\end{equation*}
Decompose the integral on the right hand side of the above display as
\begin{align*}
\int_{-\infty}^{\infty} |\phi^{(n)}_{e^{-t \sigma}, t} (t^{\frac{2}{3}} x)| \ft(t^{\frac{2}{3}} x) dx 
&=\Big(\int_{-\infty}^\sigma + \int_{\sigma}^\infty\Big)  |\phi^{(n)}_{e^{-t \sigma}, t}  (t^{\frac{2}{3}} x)| \ft(t^{\frac{2}{3}} x)dx = \EEp_1 + \EEp_2.
\end{align*}
Let us upper bound $\EEp_1$ and $\EEp_2$ respectively. By Lemma \ref{lem:phibound}, we know that  $|\phinsig(t^{\frac{2}{3}} x)| \leq C \exp(ntx)$. Using this together with Lemma \ref{lem:VnUnlaplace} \ref{item:Vnlaplace}, we get
\begin{align}\label{eq:E1p}
\EEp_1 \leq C\int_{-\infty}^\sigma e^{ntx} \f_{\frac{1}{3}, 2}(t^{\frac{2}{3}} x) dx \leq C t^{-\frac{1}{2}} \exp\big(t V_n(\sigma \wedge 4n^2)\big).
\end{align}
For $\EEp_2$, note that $\f_{\frac{1}{3}, 2}(t^{\frac{2}{3}}x)$ simplifies to $\exp(-\frac{1}{3} t x^{\frac{3}{2}})$ for $x \geq 0$. By Lemma \ref{lem:phibound}, $|\phinsig(t^{\frac{2}{3}} x)| \leq C \exp(nt\sigma)$. Using this inequality implies
\begin{align}\label{eq:E2p}
\EEp_2 = \int_{\sigma}^\infty \phinsig(t^{\frac{2}{3}} x) e^{-\frac{1}{3} t x^{\frac{3}{2}}} dx\leq C e^{t n\sigma} \int_{\sigma}^{\infty} e^{-\frac{1}{3}t x^{\frac{3}{2}}} dx \leq
C t^{-\frac{2}{3}} e^{t n\sigma - \frac{1}{3} t \sigma^{\frac{3}{2}}}.
\end{align}
Recall that $\int_{-\infty}^{\infty} |\phi^{(n)}_{e^{-t \sigma}, t} (x)| \ft(x) dx = t^{\frac{2}{3}} (\EEp_1 + \EEp_2)$, combining \eqref{eq:E1p} and \eqref{eq:E2p} yields 
\begin{equation*}
\int_{-\infty}^\infty |\phi_{e^{-t \sigma}, t}^{(n)} (x)| \f_{\frac{1}{3}, 2} (x) dx \leq C t^{\frac{2}{3}} \Big(t^{-\frac{1}{2}} e^{t V_n (\sigma \wedge 4n^2)} + t^{-\frac{2}{3}} e^{t (n\sigma - \frac{1}{3} \sigma^{\frac{3}{2}})}\Big)\leq C t^{\frac{1}{6}} \exp\big(t V_n (\sigma \wedge 4n^2)\big).
\end{equation*}
The last inequality above is due to $\VV_n(\sigma \wedge 4n^2) \geq \VV_n (\sigma) = n\sigma - \frac{1}{3} \sigma^{\frac{3}{2}}$. This completes the proof of \ref{item:g1n}. 
\smallskip
\\
The proof for \ref{item:g21}, \ref{item:g2n} follows a rather similar argument as for \ref{item:g11}, \ref{item:g1n}. Instead of using Lemma \ref{lem:VnUnlaplace} \ref{item:Vnlaplace}, one needs tp Lemma \ref{lem:VnUnlaplace} \ref{item:Unlaplace}. We omit the details here.
\end{proof}
 For any $L \geq 2$ and $\vec{m} = (m_1, \dots, m_L) \in \M(L, n)$, denote by 
\begin{equation}\label{eq:im}
\II_{\vec{m}} = \int_0^1 s^{-\alpha} ds \int_{\RL} \text{Pf}\big[K(x_i, x_j)\big]_{i, j=1}^L \prod_{i=1}^L \phi^{(m_i)}_{s, t} (x_i) dx_i.
\end{equation}
\begin{prop}\label{prop:Im}
Fix $p \in \Zn$ and  $t_0 > 0$. Recall that $n = \lfloor p \rfloor + 1$ and $\alpha = p+1 - n$. Define $\dd = \min(\frac{2}{3}, \frac{p^3}{4})$. There exists $C = C(n, t_0)$ such that for all $L \geq 2$, $t \geq t_0$ and $\vec{m} \in \M(L, n)$, we have
$|\II_{\vec{m}}| \leq C^L (2L)^{\frac{L}{2}} t^{\frac{L}{6}} e^{\frac{p^3}{3} t - \dd t}$.
\end{prop}

\begin{proof}
Without loss of generality, we assume that $m_1, \dots, m_r > 0$ and $m_{r+1} = \dots  = m_L = 0$, note that $1 \leq r \leq n \wedge L$. Referring to \eqref{eq:im}, by a change of variable $s = e^{-t\sigma}$, we obtain
\begin{equation*}
\II_{\vec{m}} = \int_0^\infty e^{t(\alpha -1) \sigma} d\sigma \int_{\RL} \pfll \prod_{i=1}^L \phi^{(m_i)}_{e^{-t\sigma}, t} (x_i) dx_i.
\end{equation*}
It suffice to show that the right hand side of the above display is upper bounded by $C^L (2L)^{\frac{L}{2}} t^{\frac{L}{6}} e^{\frac{p^3 t}{3} - \delta_p t}$. We divide our argument into two stages. We prove the inequality for $L \geq 2n^3$ in \emph{Stage 1} and \emph{Stage 2} will cover the case $2 \leq L < 2n^3$. 
\smallskip
\\
\emph{Stage 1.} $L \geq 2n^3$.
Via Proposition \ref{prop:pfbound} \ref{item:pfbound1}, $\big|\pfll\big|$ is upper bounded by $C^L (2L)^{\frac{L}{2}} \prod_{i=1}^L \f_{\frac{1}{3}, 2}(x_i)$, thus
\begin{align}\label{eq:im2}
|\II_{\vec{m}}| 
\leq C^L (2L)^{\frac{L}{2}} \int_0^{\infty} e^{t(\alpha - 1) \sigma} \prod_{i=1}^L \Big(\int_{\R} \f_{\frac{1}{3}, 2} (x) |\phi_{e^{-t\sigma}, t}^{(m_i)}(x)| dx\Big) d\sigma.
\end{align}
Applying Proposition \ref{prop:integrationbound} \ref{item:g11} and \ref{item:g1n}, there exists a constant $C=  C(n, t_0)$ such that for $t > t_0$,
\begin{equation*}
\int_{\mathbb{R}} F_{\frac{1}{3}, 2}(x)|\phi_{e^{-t\sigma}, t}^{(m_i)} (x)| dx \leq
\begin{cases}
C t^{\frac{1}{6}} \exp(t V_{m_i} (\sigma \wedge 4m_i^2) - t\sigma) & i \leq r
\\
C t^{\frac{1}{6}} \exp(t V_1 (\sigma \wedge 4) - t\sigma) & i > r.
\end{cases}
\end{equation*}
Applying this inequality to the right hand side of \eqref{eq:im2}, we find that $|\II_{\vec{m}}| \leq C^L (2L)^{\frac{L}{2}} t^{\frac{L}{6}} \int_0^\infty e^{t \MMone(\sigma)} d\sigma$, 
where $$\MMone(\sigma) = (\alpha -1)\sigma +  \sum_{i=1}^r\VV_{m_i}(\sigma \wedge 4m_i^2) + (L-r)(\VV_{1} (\sigma \wedge 4) - \sigma).$$  
To prove Proposition \ref{prop:Im}, it suffices to show  that for there exists $C = C(n, t_0)$, such that for all $t \geq t_0$ and $\vec{m} \in \mathfrak{M}(L, n)$,
\begin{equation}\label{eq:M1}
\int_0^\infty e^{t \MMone (\sigma)} d\sigma \leq Ce^{\frac{p^3}{3} t - \dd t}.
\end{equation}
where $\delta_p = \min(\frac{2}{3}, \frac{p^3}{4})$. To this aim, we decompose 
$$\int_0^\infty e^{t\MMone (\sigma)} d\sigma = \int_0^{4} e^{t \MMone(\sigma)} d\sigma + \int_{4}^\infty e^{t\MMone(\sigma)} d\sigma = \JJ_1 + \JJ_2.$$
For $\JJ_1$, since $\sigma \leq 4$, $\MMone$ simplifies to 
\begin{align*}
\MMone(\sigma) &= (\alpha- 1) \sigma + \sum_{i=1}^r \VV_{m_i} (\sigma) + (L-r) (\VV_1(\sigma)  - \sigma) \\
&= (\alpha - 1) + \sum_{i=1}^r (m_i \sigma - \frac{1}{3} \sigma^{\frac{3}{2}}) - \frac{1}{3}(L-r) \sigma^{\frac{3}{2}} = p\sigma - \frac{L}{3} \sigma^{\frac{3}{2}}.
\end{align*}
The last equality is due to $\sum_{i=1}^r m_i + \alpha -1 = n+\alpha -1 = p$.
 Since $L \geq 2n^3 \geq 2$, $\MMone(\sigma) = p\sigma - \frac{2L}{3} \sigma^{\frac{3}{2}} \leq p\sigma - \frac{4}{3} \sigma^{\frac{3}{2}} \leq \frac{p^3}{12}$. We find that 
\begin{equation}\label{eq:Jbound1}
\JJ_1 = \int_0^4 e^{t \MMone(\sigma)} d\sigma \leq \int_0^4 e^{\frac{p^3 t}{12}} d\sigma  \leq 4 e^{\frac{p^3 t}{12}}.
\end{equation}
For $\JJ_2$, since $\sigma \geq 4$, $\VV_1 (\sigma \wedge 4) = \VV_1 (4) = \frac{4}{3}$. Moreover, the maximum of $\VV_{m_i}(\sigma) = m_i \sigma - \frac{1}{3} \sigma^{\frac{3}{2}}$ equals $\frac{4}{3} m_i^3$, hence $\VV_{m_i} (\sigma \wedge 4m_i^2) \leq \frac{4}{3} m_i^3$. As a result,  
$$\MMone(\sigma)\leq (\alpha - 1)\sigma + \frac{4}{3} \sum_{i=1}^r m_i^3  + (L-r) (\frac{4}{3} - \sigma) \leq (\alpha - 1) \sigma + \frac{4}{3} n^3  - (L-r) \frac{8}{3}.$$
The last inequality follows from the fact that $\sum_{i=1}^r m_i^3 \leq (\sum_{i=1}^r m_i)^3 = n^3$ and $\frac{4}{3} - \sigma \leq -\frac{8}{3}$. Note that $r$ is the number of $m_i$ which is non-zero, so $r \leq n$. Moreover, since $L \geq 2n^3$,
$$\MMone(\sigma) \leq \frac{4}{3} n^3 - (2n^3-n) \frac{8}{3} + (\alpha - 1) \sigma \leq  (\alpha- 1)\sigma$$ Consequently, we have 
$
\JJ_2 = \int_4^\infty e^{t \MMone (\sigma)} d\sigma \leq \int_0^{\infty} e^{(\alpha -1) \sigma t} d\sigma = \frac{t^{-1}}{1-\alpha}.
$
Combining this with \eqref{eq:Jbound1} yields that for all $t > t_0$, (note that $\frac{p^3}{12} \leq \frac{p^3}{3} - \dd$)
\begin{equation*}
\int_0^\infty e^{t \MMone (\sigma)} d\sigma = \JJ_1 + \JJ_2 \leq 4 e^{\frac{p^3 t}{12}} + \frac{t^{-1}}{1-\alpha} \leq C e^{\frac{p^3}{3} t - \dd t}.
\end{equation*}
We prove the desired \eqref{eq:M1} and conclude our proof for \emph{Stage 1}.
\smallskip
\\
\emph{Stage 2.} $2 \leq L\leq 2n^3$.
Via Proposition \ref{prop:pfbound} \ref{item:pfbound2}, $\big|\pfll\big|$ is bounded by $C^L \prod_{i=1}^L \f_{\frac{2}{3}, 2}(x_i)$. Note that we throw out the multiplier $\sqrt{2L}!$ in the upper bound since it is bounded by a constant that only depends on $n$ when $L \leq 2n^3$. Thus
\begin{align*}
|\II_{\vec{m}}| 
&= C^L \int_0^{\infty} e^{t(\alpha - 1) \sigma} \prod_{i=1}^L \Big(\int_{\R} \f_{\frac{2}{3}, 2} (x) |\phi_{e^{-t\sigma}, t}^{(m_i)}(x)| dx\Big) d\sigma
\end{align*}
Applying Proposition \ref{prop:integrationbound} \ref{item:g21}, \ref{item:g2n}. For each $\int_{\mathbb{R}} \f_{\frac{2}{3}, 2} (x) |\phi^{(m_i)}_{e^{-t\sigma, t}} (x)| dx$,  $i = 1, \dots, r$, since $m_i \in \Zp$, this integral can be upper bounded by $C t^{\frac{1}{6}} \exp\big(t U_{m_i} (\sigma \wedge m_i^2)\big)$. When $i \geq r+1$, $m_i = 0$, the integral can be upper bounded by $C t^{\frac{1}{6}} \exp\big(t U_1 (\sigma \wedge 1) - t\sigma\big) $. Therefore, there exists a constant  $C = C(n, t_0)$ such that for all $t > t_0$, $2 \leq L \leq 2n^3$ and $\vec{m} \in \mathfrak{M}(L, n)$,  
\begin{equation}\label{eq:imbound1}
|\II_{\vec{m}}| \leq C^L  t^{\frac{L}{6}} \int_0^{\infty} e^{t \MMtwo(\sigma)} d\sigma 
\end{equation}
where 
$\MMtwo(\sigma) = (\alpha - 1) \sigma + \sum_{i=1}^r \UU_{m_i} (\sigma \wedge m_i^2) + (L-r) \big(\UU_1 (\sigma \wedge 1) - \sigma\big).$   To conclude the proof of Proposition \ref{prop:Im}, it suffices to show that  there exists $C = C(n, t_0)$ such that for all $t > t_0$ and $L \geq 2$ and $\vec{m} \in \mathfrak{M}(L, n)$,
\begin{equation}\label{eq:MM2}
\int_0^{\infty} e^{t \MMtwo(\sigma)} d\sigma \leq C e^{\frac{p^3}{3} t - \dd t}. 
\end{equation}
Once this is shown, applying \eqref{eq:imbound1} completes the proof of the Proposition \ref{prop:Im}. 
\smallskip
\\
We are left to show \eqref{eq:MM2}. To this aim, we divide our argument into two cases, depending on $r = 1$ or not.
\smallskip
\\
\textbf{Case 1.} $r = 1$. In this case, $m_1 = n$ and $m_i = 0$ for $i > 1$. As a result,  \begin{equation}\label{eq:Mfunc}
\MMtwo (\sigma) = (\alpha -1)\sigma +  \UU_n(\sigma \wedge n^2) + (L-1)\big(\UU_{1} (\sigma \wedge 1) - \sigma\big).
\end{equation}
We decompose 
\begin{equation*}
\int_0^{\infty} e^{t\MMtwo(\sigma)} d\sigma = \Big(\int_0^1 + \int_1^n + \int_n^{\infty}\Big) e^{t \MMtwo(\sigma)} d\sigma =  \II_1 + \II_2 + \II_3,
\end{equation*}  
and we are going to upper bound $\II_1, \II_2, \II_3$ respectively.
\smallskip
\\
For  $\II_1$, when $\sigma \leq 1$, the right hand side of \eqref{eq:Mfunc} can be simplified as 
$$\MMtwo (\sigma) = (\alpha - L) \sigma + \UU_n(\sigma) + (L-1) \UU_1(\sigma) =  p\sigma - \frac{2L}{3} \sigma^{\frac{3}{2}}.$$ 
Since $L \geq 2$, similar to the discussion in \emph{Stage 1},   
$
\MMtwo(\sigma) = p \sigma - \frac{2L}{3} \sigma^{\frac{3}{2}} \leq p\sigma  - \frac{4}{3}  \sigma^{\frac{3}{2}} \leq \frac{p^3}{12}.
$
Thereby, 
\begin{equation}\label{eq:i1bound}
\II_1 = \int_0^{1} e^{t \MMtwo (\sigma)} d\sigma \leq  \int_0^1 e^{\frac{p^3 t}{12}} d\sigma = e^{\frac{p^3 t}{12}}.
\end{equation}
For $\II_2$, when $1 \leq \sigma \leq n^2$, We can simplify $\MMtwo (\sigma) = p\sigma - \frac{2}{3} \sigma^{\frac{3}{2}} + (L-1)\big(\frac{1}{3} - \sigma\big)$. Note that the maximum of $p\sigma - \frac{2}{3} \sigma^{\frac{3}{2}}$ under the condition $\sigma  \geq 0$ equals $\frac{1}{3} p^3$. Using this in conjunction with $L \geq 2$, 
$\MMtwo(\sigma) \leq \frac{1}{3} p^3 + \frac{1}{3} - \sigma$. As a result, 
\begin{equation}\label{eq:i2bound}
\II_2 = \int_1^{n^2} e^{t \MMtwo (\sigma)} d\sigma \leq e^{\frac{p^3}{3} t}\int_1^{\infty} e^{(\frac{1}{3} - \sigma) t} d\sigma = t^{-1} e^{\frac{p^3-2}{3} t}.
\end{equation}
For $\II_3$, the right hand side of \eqref{eq:Mfunc} simplifies to 
$$\MMtwo(\sigma) = (\alpha - 1) \sigma + \UU_n(n^2) + (L-1) (\UU_1(1) - \sigma) =   (\alpha - 1)\sigma + \frac{1}{3} n^3  + (L-1)(\frac{1}{3} - \sigma).$$ 
Since $\alpha < 1$, $\sigma \geq n^2$ and $L \geq 2$, 
\begin{equation*}
\MMtwo(\sigma) \leq  n^2(\alpha - 1) +  \frac{1}{3} n^3 + (L-1)(\frac{1}{3} - \sigma) \leq n^2(\alpha - 1) +  \frac{1}{3} n^3 + (\frac{1}{3} - \sigma)
\end{equation*}
Note that $n^2(\alpha - 1) + \frac{1}{3} n^3 \leq \frac{1}{3} (n+ \alpha -1)^3 = \frac{1}{3} p^3$, hence $\MMone(\sigma) \leq \frac{1}{3} p^3 + \frac{1}{3} - \sigma$. 
Thereby, 
\begin{equation}\label{eq:i3bound}
\int_{n^2}^{\infty} e^{t \MMtwo(\sigma)} d\sigma \leq e^{\frac{p^3 t}{3}} \int_{n^2}^\infty e^{(\frac{1}{3} - \sigma)t}  d\sigma = t^{-1} e^{(\frac{p^3+1}{3}-n^2) t} \leq t^{-1} e^{\frac{p^3 - 2}{3} t}.
\end{equation}
Combining \eqref{eq:i1bound}, \eqref{eq:i2bound} and \eqref{eq:i3bound}, we conclude that for $t \geq t_0$,
\begin{equation*}
\int_0^\infty e^{t \MMtwo(\sigma)} d\sigma \leq  \Big( e^{\frac{1}{12} p^3 t} +  t^{-1} e^{\frac{p^3-2}{3} t} + t^{-1} e^{\frac{p^3 -2}{3} t}\Big) \leq C e^{\frac{1}{3} p^3t - \delta_p t},
\end{equation*}
The last inequality follows since $\delta_p = \min(\frac{2}{3}, \frac{1}{4}p^3)$. So far we have shown \eqref{eq:MM2} when $r = 1$.
\smallskip
\\
\textbf{Case 2.} $r \geq 2$. This implies $n \geq 2$. We write 
\begin{equation*}
\int_0^{\infty} e^{t\MMtwo(\sigma)} d\sigma = \Big(\int_0^{1} + \int_1^\infty\Big) e^{t\MMtwo(\sigma)} d\sigma =  \IIt_1 + \IIt_2
\end{equation*} 
For $\IIt_1$, when $\sigma \leq 1$, $\MMtwo(\sigma) = p\sigma - \frac{2}{3} L \sigma^3$. Via the same argument as in \textbf{Case 1},  we conclude that $\IIt_1 \leq  e^{\frac{1}{12} p^3 t}$.
For $\IIt_2$,  using the inequality  $\UU_{m_i}(\sigma \wedge m_i^2) \leq \frac{1}{3} m_i^3 $ and $\UU_1(\sigma \wedge 1) = \UU_1 (1) = \frac{1}{3}$, we get
\begin{align}\label{eq:IIt2}
\MMtwo(\sigma) &\leq \frac{1}{3} \sum_{i=1}^r  m_i^3 + \frac{1}{3}(L-r)  + (\alpha - L +r -1) \sigma \leq \frac{1}{3} \sum_{i=1}^r m_i^3 + (\alpha - 1) \sigma
\end{align} 
Since we assume $r \geq 2$, it is convincible that $\sum_{i=1}^r \ m_i^3$ is at most $(n-1)^3 + 1$, since the cubic sum will increase if we let mass concentrate on fewer terms. To justify this, note that $\sum_{i=2}^r m_i^3 \leq \big(\sum_{i=2}^r m_i\big)^3 = (n-m_1)^3$. Thus $$\sum_{i=1}^r m_i^3 \leq m_1^3 + (n-m_1)^3 = (n-1)^3 + 1 + 3n(m_1 - 1)(m_1 - (n-1)) \leq (n-1)^3 + 1.$$
Applying this inequality to the right hand side of \eqref{eq:IIt2}, we see that 
$$\MMtwo(\sigma) \leq \frac{1}{3} \big((n-1)^3 + 1\big) + (\alpha - 1) \sigma \leq \frac{1}{3} (n-1)^3 + \alpha - \frac{2}{3} \leq \frac{1}{3} (\alpha +n -1)^3 - \frac{2}{3}.$$
The second inequality above follows from $\sigma \geq 1$ and the third equality is due to $\frac{1}{3} (n-1)^3 + \alpha \leq \frac{1}{3} (n-1)^3 + (n-1)^2 \alpha \leq 
\frac{1}{3} (n-1 + \alpha)^3$.
Recall that $p  = \alpha + n- 1$, we obtain $\MMtwo(\sigma) \leq \frac{1}{3} p^3 - \frac{2}{3} + (\alpha -1) \sigma$, and thus $\IIt_2 \leq  \int_1^\infty \exp\big(t( \frac{1}{3} p^3 - \frac{2}{3} + (\alpha -1) \sigma)\big) d\sigma = \frac{1}{1-\alpha} e^{(\frac{1}{3} p^3 - \frac{2}{3}) t}$.
So there exists a constant $C$ such that for $t > t_0$, 
$$\int_0^{\infty} e^{t \MMtwo (\sigma)} d\sigma =\IIt_1 + \IIt_2 \leq e^\frac{p^3 t}{12} + (1-\alpha)^{-1} e^{\frac{p^3-2}{3} t} \leq Ce^{\frac{p^3}{3} t -\dd t}.$$
This implies \eqref{eq:MM2} and completes the proof of \emph{Stage 2}.
\end{proof}
\begin{proof}[Proof of Proposition \ref{prop:minor}]
It suffices to prove that for fixed $p > 0$, there exists a constant $C  = C(p)$ such that for all $L \geq 2$ and $t > 1$,
\begin{equation}\label{eq:BpLbound}
|\BBpL(t)| \leq \frac{C^L (2L)^{\frac{L}{2}}}{L!} t^{\frac{L}{6}} e^{\frac{p^3}{3} t - \delta_p t}.
\end{equation}
One this is shown, we conclude our proof by observing
\begin{equation}\label{eq:temp6}
\Big|\sum_{L=2}^{\infty} \BBpL(t)\Big| \leq \sum_{L=2}^\infty |\BBpL(t)| \leq e^{\frac{p^3}{3} t - \delta_p t} \sum_{L=2}^\infty \frac{C^L (2L)^{\frac{L}{2}} t^{\frac{L}{6}}}{L!}
\end{equation}
Using the inequality of Stirling's formula, ew know that $L^L \leq e^L L!$ for all $L \in \Zp$. Consequently, $(2L)^{\frac{L}{2}} = 2^L (L^L)^\frac{1}{2} \leq 2^L e^{\frac{L}{2}} \sqrt{L!}$.  So there exists constant $C_1, C_2$ such that for all $t > 1$,
\begin{equation*}
\sum_{L=2}^\infty \frac{C^L (2L)^{\frac{L}{2}} t^{\frac{L}{6}}}{L!} \leq \sum_{L=2}^\infty \frac{C_1^L t^{\frac{L}{6}}}{\sqrt{L!}} \leq e^{C_2 t^{\frac{1}{3}}}.
\end{equation*}
Combining the above inequality with \eqref{eq:temp6}, the left hand side of \eqref{eq:temp6} is upper bounded by $\exp(p^3 t/3 - \delta_p t + C_2 t^{\frac{1}{3}})$. 
Taking the logarithm and dividing by $t$ for both sides and letting $t \to \infty$ completes the proof of Proposition \ref{prop:minor}.
\smallskip
\\
We are left to show \eqref{eq:BpLbound}. referring to \eqref{eq:bpL} and using Leibniz's rule,
\begin{equation}\label{eq:BpL1}
\mathcal{B}_{p, L} (t) = \frac{(-1)^n}{\Gamma (1 - \alpha)}  \sum_{\vec{m} \in \M(L, n)}\int_0^{1} s^{-\alpha} \frac{1}{L!} \int_{\RL} \pfll   \prod_{i=1}^L \phi^{(m_i)}_{s, t}(x_i) dx
=  \frac{(-1)^n}{\Gamma(1-\alpha) L!} \sum_{\vec{m} \in \M(L, n)}\binom{n}{\vec{m}} \II_{\vec{m}} 
\end{equation} 
where $\II_{\vec{m}}$ is defined in \eqref{eq:im}.
Using  Proposition \ref{prop:Im} and the above display, recalling that  $\# \M(L, n)$ represents the number of elements that lie in $\M (L, n)$, we get for all $L \geq 2$, 
\begin{equation*}
|\BBpL(t)| \leq \frac{1}{\Gamma(1 - \alpha) L!} \sum_{\vec{m} \in \M(L, n)}\binom{n}{\vec{m}} |\II_{\vec{m}}| \leq \frac{n!}{\Gamma(1-\alpha)}   (\# \M(L, n)) \max_{\vec{m} \in \M(L, n)}|\II_{\vec{m}}|
\end{equation*}
where the first inequality follows from taking the absolute value of both sides of \eqref{eq:BpL1} and applying triangle inequality to the right hand side. The second inequality follows from upper bounding $\binom{n!}{\vec{m}}$ with $n!$. Recall from \eqref{eq:MLn} that $\# \M(L, n) \leq L^n$. To prove \eqref{eq:BpLbound}, applying Proposition \ref{prop:Im} to upper bound each $|\II_{\vec{m}}|$, we obtain
$$|\BBpL(t)| \leq \frac{n! L^n}{\Gamma(1 - \alpha) L!}  C^L (2L)^{\frac{L}{6}} t^{\frac{L}{6}} e^{\frac{p^3 t}{3} - \delta_p t}.$$
Note that $L^n$ grows slower than $C_1^L$ for $C_1 > 1$, as $L \to \infty$. So there exists a constant $C_1$ such that  $\frac{n! L^n}{\Gamma(1-\alpha)} C^L \leq C_1^L$. Applying this inequality to the right hand side of the above display   completes the proof of \eqref{eq:BpLbound}.
\end{proof}
\appendix 
\section{Basic facts of Airy function}
 In this section, we review some basic properties of the Airy function. 
\begin{lemma}\label{lem:airyproperty}
We have the following asymptotics for Airy function
\begin{align*}
\emph{Ai}(x) 
\sim 
\begin{cases}
\frac{e^{-\frac{2}{3} x^{\frac{3}{2}}}}{2 \sqrt{\pi} x^{\frac{1}{4}}} \hspace{9em} x \to +\infty,\\ 
\frac{1}{\sqrt{\pi} |x|^{\frac{1}{4}}}\cos\Big(\frac{\pi}{4} - \frac{2 |x|^{\frac{3}{2}}}{3}\Big) \qquad\, x \to -\infty.
\end{cases}
\quad
\emph{Ai}'(x) \sim 
\begin{cases}
-\frac{x^{\frac{1}{4}} e^{-\frac{2}{3} x^{\frac{3}{2}}}}{2 \sqrt{\pi}}  \hspace{8em} x \to +\infty,\\ 
-\frac{|x|^{\frac{1}{4}}}{\sqrt{\pi} }\cos\Big(\frac{\pi}{4} + \frac{2 |x|^{\frac{3}{2}}}{3}\Big) \hspace{3.3em} x \to -\infty.
\end{cases}
\end{align*}
\end{lemma}
\begin{proof}
See  Eq 10.4.59-10.4.62 of \cite{AS48}.
\end{proof}
\begin{lemma}\label{lem:airyintegral}
We have
$\int_{-\infty}^{\infty} \emph{Ai}(x) dx = 1.$ and $\int_{-\infty}^0 \emph{Ai}(x) dx = 1/3$.
\end{lemma}
\begin{proof}
See page 431 of \cite{Olver97}.
\end{proof}
\begin{lemma}\label{lem:airysquareintegral}
There exists constant $C$ such that 
\begin{align*}
&\frac{1}{C (x+1)} e^{-\frac{4}{3} x^{\frac{3}{2}}}\leq \int_0^{\infty} \emph{Ai}(x+\lambda)^2 d\lambda \leq \frac{C}{x+1} e^{-\frac{4}{3} x^{\frac{3}{2}}} \quad\, \forall\, x \geq 0 \\
\label{eq:temp14}
&\frac{1}{C} (\sqrt{|x|} + 1) \leq \int_0^{\infty} \emph{Ai}(x+\lambda)^2 d\lambda \leq C (\sqrt{|x|} + 1)  ,\qquad \forall\,x \leq 0
\end{align*}
\begin{proof}
This is Eq 2.8 and Eq 2.9 of \cite{DT19}.
\end{proof}
\end{lemma}
\section{Estimate of the Pfaffian Kernel}
 In this section, we provide various bounds for the entries in the GOE Pfaffian kernel. As a notational convention, we say $f(x) \sim g(x)$ as $x \to a$ (where $a$ can be $\pm \infty$) if $\lim_{x \to a} \frac{f(x)}{g(x)} = 1$. 
\begin{lemma}\label{lem:k12bound}
There exists a constant $C > 0$ such that 
\begin{enumerate}[leftmargin = 20pt, label = (\roman*)]
\item \label{item:k12 xpositive}$\frac{\exp(-\frac{2}{3} x^{\frac{3}{2}})}{C (1+x)^{\frac{1}{4}}} \leq K_{12}(x,x) \leq \frac{C \exp(-\frac{2}{3} x^{\frac{3}{2}})}{(1+x)^{\frac{1}{4}}}\qquad \forall\, x \geq 0,$
\item
\label{item:k12 xnegative}
$0 \leq K_{12}(x, x) \leq C\sqrt{1-x} \hspace{1.9em}  \hspace{5em} \forall\,\ x \leq 0.$ 
\end{enumerate}

\end{lemma}
\begin{proof}
We first prove \ref{item:k12 xpositive}.
By setting $x = y$ in \eqref{eq:thmk12}, we get 
\begin{equation}\label{eq:temp13}
K_{12} (x, x) = \int_0^{\infty} \text{Ai}(x+\lambda)^2 d\lambda + \frac{1}{2} \text{Ai}(x) \int_{-\infty}^x \text{Ai}(\lambda) d\lambda.
\end{equation}
For the second term in the above display, by Lemma \ref{lem:airyproperty} and Lemma \ref{lem:airyintegral}, we have as $x \to +\infty$
\begin{align*}
\text{Ai}(x) \int_{-\infty}^{x} \text{Ai}(\lambda) d\lambda \sim \frac{e^{-\frac{2}{3} x^{\frac{3}{2}}}}{2 \sqrt{\pi} x^{\frac{1}{4}}}
\end{align*}
Combining this with the first inequality of  Lemma \ref{lem:airysquareintegral}, which controls the first term on the right hand side,  the upper bound in \ref{item:k12 xpositive} naturally follows. To prove the lower bound of \ref{item:k12 xpositive}, due to the above displayed asymptotic and the non-negativity of $\int_0^\infty \text{Ai}(x+\lambda)^2 d\lambda$, there exists constant $M$ and $C$ such that for $x > M$, $$K_{12}(x, x) \geq C^{-1} x^{-\frac{1}{4}}\exp(-\frac{2}{3} x^{\frac{3}{2}}).$$  To conclude the lower bound in \ref{item:k12 xpositive}, it suffices to show that the minimum of $K_{12} (x, x)$ is positive over $[0, M]$ ($K_{12}(x, x)$ is continuous, so admits a minimum). Due to Eq. \eqref{eq:temp13} and Lemma \ref{lem:airyintegral}, we can rewrite $K_{12}(x, x) = \int_{0}^\infty \text{Ai}(x+\lambda)^2 d\lambda + \frac{1}{3} \text{Ai}(x) +  \frac{1}{2}\text{Ai} (x) \int_{0}^x \text{Ai} (\lambda) d\lambda $. Since $\text{Ai}(x)$ is positive for $x \geq 0$, this implies $K_{12}(x, x) > 0$ for all $x > 0$, which completes the proof of the lower bound.
\bigskip
\\
We move on proving \ref{item:k12 xnegative}. The lower bound follows directly since $K_{12} (x, x)$ is the first order correlation function of a Pfaffian point process, thus is negative. For the upper bound, by the asymptotic of $\text{Ai}(x)$ at $-\infty$, there exists constant $C$ such that for all $x \leq 0$,
\begin{equation*}
\bigg|\text{Ai}(x) \int_{-\infty}^{x} \text{Ai}(\lambda) d\lambda\, \bigg| \leq C (1+|x|)^{-\frac{1}{4}}.
\end{equation*}
The result then follows from the second inequality of Lemma \ref{lem:airysquareintegral} and \eqref{eq:temp13}.
\end{proof}
 Recall that we defined $\f_{\alpha, \beta}(x) = C\big( e^{-\alpha x^{\frac{3}{2}}} \mathbf{1}_{\{x \geq 0\}} + (1-x)^\beta \mathbf{1}_{\{x < 0\}}\big)$. 
\begin{lemma}\label{lem:kernelbound}
There exists a constant $C$, such that for all $x, y \in \mathbb{R}$, we have the following upper bounds for the Pfaffian kernel entries:
\begin{enumerate}[leftmargin=20pt, label=(\alph*)]
\item \label{item:k11}$|K_{11}(x, y)| \leq C \big(\f_{\frac{2}{3}, \frac{5}{4}} (x) \wedge \f_{\frac{2}{3}, \frac{3}{4}}(x) \f_{\frac{2}{3}, \frac{3}{4}}(y)\big)$
\item \label{item:k12}
$|K_{12}(x, y)| \leq C \big(\f_{\frac{2}{3}, \frac{3}{4}}(x) \wedge  \f_{0, \frac{3}{4}} (y)\big)$
\item \label{item:k22}
$|K_{22} (x, y)| \leq C \f_{0, \frac{3}{4}} (x) $
\end{enumerate}
\end{lemma}
\begin{proof}	
For \ref{item:k11}, it suffices to show that $|K_{11}(x, y)| \leq C \f_{\frac{2}{3}, \frac{5}{4}} (x)$ and  $|K_{11}(x, y)| \leq C\f_{\frac{2}{3}, \frac{3}{4}}(x) \f_{\frac{2}{3}, \frac{3}{4}}(y)$.
Recall the expression of $K_{11}(x, y)$ from \eqref{eq:thmk11}. Using integration by parts for the right hand side of  \eqref{eq:thmk11}, we get
$K_{11}(x, y) = \text{Ai}(x) \text{Ai}(y) -  2\int_0^{\infty} \text{Ai}(y + \lambda) \text{Ai}'(x+\lambda) d\lambda.$ This implies that $|K_{11}(x, y)| \leq |\text{Ai}(x) \text{Ai}(y)| + 2 \int_0^\infty |\text{Ai}(y + \lambda) \text{Ai}'(x+ \lambda)| d\lambda$. Since $|\text{Ai}(x)|$ is a bounded function, there exists constant $C$ such that
\begin{align*}
|K_{11}(x, y)| \leq C|\text{Ai}(x)| + C \int_0^{\infty} |\text{Ai}'(x+\lambda)| d\lambda =  C + C \int_x^{\infty} |\text{Ai}'(\lambda)| d\lambda.
\end{align*}
To obtain the upper bound for $|\text{Ai}(x)|$ and $\int_x^{\infty} |\text{Ai}(\lambda)| d\lambda$, it suffices to look at their behavior as $x \to \pm \infty$. The asymptotic  $\text{Ai}'(x)$ at $\pm \infty$ is specified in Lemma \ref{lem:airyproperty}. Therefore, 
\begin{equation*}
\int_x^\infty |\text{Ai}'(\lambda)| d\lambda \leq C e^{-\frac{2}{3} x^{\frac{3}{2}}}, \text{ if } x \geq 0; \qquad \int_x^\infty |\text{Ai}'(\lambda)| d\lambda \leq C(1 - x)^{\frac{5}{4}} \text{ if } x \leq 0. 
\end{equation*}
This implies that $|K_{11}(x, y)| \leq C \f_{\frac{2}{3}, \frac{5}{4}} (x)$.
In addition, since $$K_{11}(x, y) = \int_0^{\infty} \text{Ai}(x+\lambda) \text{Ai}'(y+\lambda) d\lambda - \int_0^{\infty} \text{Ai}'(x+\lambda) \text{Ai}(y+\lambda) d\lambda = A_1 - A_2$$ By Cauchy Schwartz inequality, 
\begin{equation*}	
A_1^2 \leq \int_0^{\infty} \text{Ai}(x + \lambda)^2 d\lambda \int_0^{\infty} \text{Ai}'(y + \lambda)^2 d\lambda = \int_x^{\infty} \text{Ai}(\lambda)^2 d\lambda \int_{y}^{\infty} \text{Ai}'(\lambda)^2 d\lambda.
\end{equation*}
By Lemma \ref{lem:airyproperty}, $\Ai(x)^2$ decays asymptotically as $\exp(-\frac{4}{3} x^{\frac{3}{2}})$ as $x \to +\infty$ and is asymptotically upper bounded by $|x|^{-\frac{1}{2}}$ as $x \to -\infty$. This implies that  $\int_x^{\infty} \text{Ai}(\lambda)^2 d\lambda \leq C F_{\frac{4}{3}, \frac{1}{2}} (x)$. Similarly, $\text{Ai}'(y)^2$ decays asymptotically as $\exp(-\frac{4}{3} y^{\frac{3}{2}})$ and is asymptotically upper bounded by $|y|^{\frac{1}{2}}$, we get $\int_y^{\infty} \text{Ai}'(\lambda)^2 d\lambda \leq C F_{\frac{4}{3}, \frac{3}{2}}(y)$. As a result, 
\begin{align*}
|A_1| &\leq \Big(\int_x^{\infty} \text{Ai}(\lambda)^2 d\lambda\Big)^{\frac{1}{2}} \Big(\int_y^{\infty} \text{Ai}'(\lambda)^2 d\lambda\Big)^{\frac{1}{2}} \leq C \f_{\frac{2}{3}, \frac{1}{4} }(x) \f_{\frac{2}{3}, \frac{3}{4} }(y) \leq C \f_{\frac{2}{3}, \frac{3}{4} }(x) \f_{\frac{2}{3}, \frac{3}{4} }(y).
\end{align*}
For the second inequality above, we use the property that $\sqrt{\f_{\alpha, \beta}} = F_{\alpha/2, \beta/2} $ and for the third inequality, $\f_{\alpha, \beta}(x)$ is increasing in $\beta$. Interchanging the role of $x$ and $y$, we also have  $|A_2| \leq  C \f_{\frac{2}{3}, \frac{3}{4} }(x) \f_{\frac{2}{3}, \frac{3}{4} }(y)$. Therefore, the same upper bound holds for $|K_{11} (x, y)|$ and we conclude the proof of \ref{item:k11}.
\bigskip
\\
We move on showing \ref{item:k12}. We will prove $|K_{12}(x, y)| \leq C \f_{\frac{2}{3}, \frac{3}{4}}(x)$ and  $|K_{12}(x, y)| \leq C  \f_{0, \frac{3}{4}} (y)$ respectively.
Recall
$K_{12} (x, y)$ from \eqref{eq:thmk12}. Note that both $|\text{Ai}(y + \lambda)|$ and  $|\int_{-\infty}^y \text{Ai}(\lambda) d\lambda|$ are bounded function of $y$ (see Lemma \ref{lem:airyintegral}), by using triangle inequality,
\begin{align*}
|K_{12} (x, y)| \leq \frac{1}{2} \int_0^{\infty} |\text{Ai}(x+\lambda) \text{Ai} (y+\lambda)| d\lambda + \frac{1}{2} |\text{Ai}(x)| \cdot \Big|\int_{-\infty}^y \text{Ai}(\lambda) d\lambda\Big|
\leq C \int_x^{\infty} |\text{Ai}(\lambda)| d\lambda + C |\text{Ai}(x)|.
\end{align*}
By the asymptotic of $\text{Ai}(x)$ at $\pm \infty$, (use the similar approach as in part \ref{item:k11}), we see that $|K_{12} (x, y)| \leq C \f_{\frac{2}{3}, \frac{3}{4}} (x).$
\smallskip
We proceed to obtain a different upper bound for $K_{12}$. Referring to the right hand side of the first inequality in the above display and upper bounding $|\text{Ai}(x +\lambda)|$ and $|\frac{1}{2} \text{Ai}(x) \int_{-\infty}^y \text{Ai}(\lambda) d\lambda|$ by a constant, we find that
\begin{align*}
|K_{12}(x, y)| 
\leq C \int_0^{\infty} |\text{Ai}(y+\lambda)| d\lambda + C \leq C \f_{0, \frac{3}{4}}(y).
\end{align*}
This concludes our proof of \ref{item:k12}.
\bigskip
\\
Finally, let us demonstrate \ref{item:k22}. Recall from \eqref{eq:thmk22} that
\begin{align*}
K_{22} (x, y) &= \frac{1}{4} \int_0^\infty \text{Ai}(x + \lambda) \Big(\int_\lambda^\infty  \text{Ai}(y+\mu)d\mu\Big) d\lambda   - \frac{1}{4} \int_0^{\infty} \text{Ai}(y+ \lambda) \Big(\int_{\lambda}^{\infty} \text{Ai}(x+\mu)  d \mu\Big) d\lambda \\
\numberthis \label{eq:prop6 k22} 
& \quad - \frac{1}{4} \int_0^\infty \text{Ai}(x+\lambda) d\lambda + \frac{1}{4} \int_0^{\infty} \text{Ai}(y+\lambda) d\lambda - \frac{\text{sgn}(x-y)}{4}
\end{align*}
and recall that $\text{sgn}$ is the sign function. By Fubini's theorem,
\begin{align*}
\int_0^\infty\!\! \text{Ai}(y+\lambda) \Big(\int_{\lambda}^\infty \!\!  \text{Ai}(x+\mu)d\mu\Big)  d\lambda   = \Big(\int_0^{\infty} \!\! \text{Ai}(x+\lambda) d\lambda\Big) \Big(\int_0^{\infty} \!\! \text{Ai}(y+\lambda) d\lambda\Big) - \int_0^\infty\!\! \text{Ai}(x + \lambda)\Big(\int_\lambda^\infty \!\! \text{Ai}(y + \mu) d\mu\Big)  d\lambda 
\end{align*}
Replacing the term $\int_{0}^\infty \text{Ai}(y + \lambda) \big(\int_{\lambda}^{\infty} \text{Ai}(x+ \mu) d\mu\big) d\lambda$ in \eqref{eq:prop6 k22} with the right hand side in the above display, 
\begin{align*}
K_{22} (x, y)  &= \frac{1}{2} \int_0^\infty \text{Ai}(x + \lambda)  \Big(\int_{\lambda}^\infty  \text{Ai}(y + \mu)d\mu \Big)d\lambda - \frac{1}{4} \Big(\int_0^{\infty}  \text{Ai}(x+\lambda) d\lambda\Big) \Big(\int_0^{\infty}  \text{Ai}(y+\lambda) d\lambda\Big)\\
&\quad -\frac{1}{4} \int_0^{\infty} \text{Ai}(x+\lambda) d\lambda + \frac{1}{4} \int_0^\infty \text{Ai}(y + \lambda) d\lambda  - \frac{\text{sgn}(x-y)}{4}.
\end{align*}
We know that $\big|\int_0^{\infty} \text{Ai}(x+\lambda) d\lambda\big|, \big|\int_0^{\infty} \text{Ai}(y+\lambda) d\lambda\big|$ can upper bounded by a constant. Applying triangle inequality to the above display,
$$|K_{22} (x, y)| \leq C \int_0^{\infty} |\text{Ai}(x + \lambda)| d\lambda + C.$$
Using the asymptotic of $\Ai(x)$ at $\pm \infty$ in Lemma \ref{lem:airyproperty}, we find that $|K_{22}(x, y) | \leq C \f_{0, \frac{3}{4}} (x)$, thus conclude \ref{item:k22}.
\end{proof}
\bibliographystyle{alpha}
\bibliography{ref.bib}
\end{document}